\documentclass[11pt]{article}

\usepackage[utf8]{inputenc}
\usepackage{amsmath}
\usepackage{amsfonts}
\usepackage{amssymb}
\usepackage{amsthm}
\usepackage{bbm}
\usepackage{cancel}
\usepackage{color, comment}
\usepackage[square,sort,comma,numbers]{natbib}
\usepackage[colorlinks,citecolor=blue]{hyperref}
\usepackage{tikz}
\usepackage[a4paper,textwidth=14cm]{geometry}
\usepackage{subcaption}
\usepackage{listofitems}
\usepackage{pgfplots}
\pgfplotsset{compat=1.15}
\usepackage{mathrsfs}
\usetikzlibrary{arrows}

\definecolor{ccqqqq}{rgb}{0.8,0.,0.}
\definecolor{qqqqff}{rgb}{0.,0.,1.}
\definecolor{xfqqff}{rgb}{0.4980392156862745,0.,1.}

\numberwithin{equation}{section}

\renewcommand{\H}{\mathcal{H}}
\newcommand{\E}{\mathcal{E}}
\newcommand{\F}{\mathcal{F}}
\newcommand{\G}{\mathcal{G}}
\renewcommand{\L}{\mathcal{L}}
\newcommand{\N}{\mathcal{N}}
\renewcommand{\S}{\mathbb{S}}
\newcommand{\T}{\mathcal{T}}
\newcommand{\R}{\mathbb{R}}
\newcommand{\Z}{\mathbb{Z}}

\newcommand*{\genbf}[1]{\ifmmode\mathbf{#1}\else\textbf{#1}\fi}

\newcommand{\1}{\mathbbm{1}}
\renewcommand{\L}{\mathcal{L}}

\newtheorem{theorem}{Theorem}[section]
\newtheorem{lemma}[theorem]{Lemma}
\newtheorem{proposition}[theorem]{Proposition}
\newtheorem{corollary}[theorem]{Corollary}
\newtheorem{definition}[theorem]{Definition}

\theoremstyle{definition}
\newtheorem{remark}[theorem]{Remark}
\newtheorem{example}[theorem]{Example}

\newcommand{\todo}[1]{\vspace{5mm}\par \noindent
\framebox{\begin{mihttps://www.overleaf.com/project/5da4ac389b02480001849b84nipage}[c]{0.95 \textwidth} \tt #1\end{minipage}}
\vspace{5mm} \par}

\title{Continuum limits of discrete isoperimetric problems and Wulff shapes in lattices and quasicrystal tilings}
\author{Giacomo Del~Nin
\footnote{Giacomo.Del-Nin@warwick.ac.uk, Mathematics Institute, University of Warwick, 
Zeeman Building, CV4 7HP Coventry, UK. ORCID: 0000-0001-7308-9753}
 \and Mircea Petrache
\footnote{mpetrache@mat.uc.cl, Facultad de Matem\'aticas, Pontificia Universidad Cat\'olica de Chile, Avda. Vicu\~na Mackenna 4860, Macul, Santiago, 6904441, Chile. ORCID: 0000-0003-2181-169X}
}

\date{\today}

\begin{document}

\maketitle

\begin{abstract}
We prove discrete-to-continuum convergence of interaction energies defined on lattices in the Euclidean space (with interactions beyond nearest neighbours) to a crystalline perimeter, and we discuss the possible Wulff shapes obtainable in this way.
Exploiting the “multigrid construction” of quasiperiodic tilings (which is an extension of De Bruijn's “pentagrid” construction of Penrose tilings) we adapt the same techniques to also find the macroscopical homogenized perimeter when we microscopically rescale a given quasiperiodic tiling.
\end{abstract}

{\small
}

\setcounter{tocdepth}{2}
\tableofcontents






\section{Introduction}
The question of what crystal shapes are induced by what kind of interactions has preoccupied researches since the beginning of the field of crystallography. Mathematically, the study of crystal shapes has been first put on a firm ground within the continuum theory, starting with the work of Wulff \cite{wulff}, later reformulated and extended by Herring \cite{herring} and others \cite{liebmann, laue, dinghas}; see also \cite{Tay78} and references therein, for the connection to anisotropic perimeter functionals. In the continuum study, the role of the microscopic structure of the material considered is not modelled explicitly, and one starts by studying a surface energy of the form 
\begin{equation}\label{eq:defper}
P_\phi(E):=\begin{cases}
\displaystyle \int_{\partial^* E} \phi(\nu_E)d\H^{d-1} & \text{if $E$ is a set of finite perimeter,}\\[3mm]
+\infty & \text{otherwise,}
\end{cases}
\end{equation}
where $\phi:\mathbb R^d\to [0, +\infty]$ is a $1$-homogeneous convex function, $E\subset\mathbb R^d$ is a finite-perimeter set, $\partial^*E$ is the reduced boundary of $E$ and $\mathcal H^{d-1}$ is the $(d-1)$-dimensional Hausdorff measure in $\mathbb R^d$. See the book \cite{maggibook} for details. The optimizer of $P_\phi$ among unit-volume competitors gives the shape of an ideal crystal with anisotropy $\phi$, and its shape is called the {\bf Wulff shape} corresponding to $\phi$, see \cite{herring, Tay78}.

\medskip

In this work, we focus on the link between discrete energy-minimization models and the minimization giving rise to the Wulff-shape problem. We will think of a discrete crystal to be a fixed-cardinality minimum-energy subset of the vertices of a lattice, or a subset of tiles in a quasiperiodic tiling. The energies that we consider are sums of pairwise interactions that respect the periodic or quasiperiodic structure. We establish compactness and $\Gamma$-convergence results in which, when we scale down the lattice as we increase the cardinality of point or tile configurations, the discrete energy functionals converge to a perimeter functional as in \eqref{eq:defper}. We then consider the effect of modifying the discrete interaction model at the microscopic scale, on the macroscopic limit Wulff shape obtained in the $\Gamma$-limit. This endeavor fits within the general theory of discrete-to-continuum limits for crystals and quasicrystals. See \cite{AFS} for the triangular lattice, \cite{braidesetal1} for another approach for quasicrystals and \cite{braidesetal2} for a homogenization result on the Penrose tiling, and the discussion below for more related results.

\subsection{Setting and main results for lattice energies}\label{ssec:persetup}
We consider a lattice $\L=M\mathbb Z^d\subset \R^d$, where $M\in \mathrm{GL}(d)$. We define $\det\L$ as $|\det M|$ whenever $\L=M\Z^d$ (we refer to \cite{bourbaki1998general} for a proof that this is a well-defined quantity). We consider configurations $X_N=\{x_1,\ldots,x_N\}$ lying in $\L$ and 
we define the energy
\begin{equation}\label{eq:energy}
\E(X_N)=\sum_{x\in X_N} \sum_{x'\in X_N\setminus \{x\}} V(x'-x).
\end{equation}
Here $V:\mathcal L\to (-\infty,0]$ is a fixed potential which may quantify the fall in energy due to the formation of atomic bonds in a crystal, for example. We first consider the case where $V$ vanishes outside of a finite subset $\N\subset\mathcal L$ such that $\mathrm{span}_\Z\,\N=\L$ (see Section \ref{subsec:lowerdim} for more general cases).
\medskip
We will be interested in the surface-type energy
\begin{equation}\label{eq:rescenergy}
\F(X_N)=-\sum_{x\in X_N}\sum_{x'\in \L\setminus X_N}V(x'-x),
\end{equation}
which counts the energy excess due to missing bonds. Indeed the energy \eqref{eq:energy} rewrites as $\E(X_N)=C_{\E}N +\F(X_N)$, where $C_{\E}=\sum_{w\in\mathcal L\setminus\{0\}}V(w)$ is a “bulk” term independent of the shape of $X_N$. To every configuration $X\subset\L$ (not necessarily having $N$ points) we associate, denoting by $U_\L:=M([0,1)^d)$ the fundamental cell of $\L$, the set
\begin{equation}\label{eq:ENdef}
E_N(X):={N^{-\frac{1}{d}}}\bigcup_{x\in X} (x+U_\L). 
\end{equation}
We then define the rescaled energies

\begin{equation}\label{rescenergy}
\F_N(E):=\left\{
\begin{array}{ll}
\F(X_N)&\mbox{ if }\exists X_N\subset\L,\quad \sharp X_N=N,\ E_N(X_N)=E, \\[3mm]
+\infty&\mbox{ else.}\end{array}
\right.
\end{equation}

We consider the following convergence: given a sequence $(X_N)_{N\in\mathbb{N}}$, we say that $X_N$ converges to a set $E\subset \R^n$ if $E_N(X_N)\to E$ locally in measure (also referred to as the “$L^1_{loc}$ convergence”, identifying sets with their characteristic function).

\begin{theorem}[Gamma convergence for crystals]\label{thm:main}
The functionals $N^{-\frac{d-1}{d}}\F_N$ $\Gamma$-converge, with respect to the topology above, to a functional of the form $P_V:=P_{\phi_V}$ as in\eqref{eq:defper}, where 
\begin{equation}\label{eq:phiV}
\phi_V(\nu):=\frac{1}{\det \mathcal{L}}\sum_{v\in\N} |V(v)|\langle v, \nu\rangle_+.
\end{equation}
\end{theorem}

\medskip
As we discovered after the completion of the preliminary version of this paper, this result had been already proven by Gelli in her PhD thesis \cite{Gel} in a more general form (see also \cite{BraGel} and \cite{AliGel}). We decided to leave the proof (even if the ideas are very close to those in \cite{Gel}) because in our simplified case some of the intricacies of the general case are not present, and moreover the argument will be referenced later in the proof of the quasicrystal case given by Theorem \ref{thmqc}.
In order to prove Theorem \ref{thm:main} we first show in Section \ref{sec:reduction} that it is sufficient to prove it when the lattice $\L$ is $\Z^d$. Then in the rest of Section \ref{sec:proof} we prove it for $\L=\Z^d$.
We also note that for every finite perimeter set $E$ in $\R^d$ we have
\[
P_V(E)=P_{V^{sym}}(E)
\]
where $V^{sym}(v)=\tfrac12\big(V(v)+V(-v)\big)$, as proved in Proposition \ref{prop:equivsymm}, so that it is not restrictive to assume that $V$ is symmetric. We also prove the following compactness result, which motivates the chosen convergence.

\begin{proposition}[Compactness]\label{prop:compactness}
Suppose that $\mathrm{span}_\Z\,\N=\L$. Given a sequence $X_N$ such that
\[
\F(X_N)\leq C N^{\frac{d-1}{d}} 
\]
there exists a subsequence $X_{N_k}$ and a finite perimeter set $E$ such that $E_{N_k}(X_{N_k})\to E$ in $L^1_{loc}$.
\end{proposition}

As a consequence of Theorem \ref{thm:main} and Proposition \ref{prop:compactness} we obtain the following.

\begin{corollary}\label{cor:Wulff}
Minimizers of $\F_N$ converge locally in measure, up to rescaling and possibly a translation, to a finite perimeter set $E$ that minimizes \eqref{eq:defper} for its own volume constraint. 
More precisely (as explained in Section \ref{sec:wshapes}), in the case of the anisotropy \eqref{eq:phiV} this Wulff shape coincides with the Minkowski sum of segments given by 
\begin{equation}\label{ws13}
\mathcal{W}_{\phi_V}=\sum_{v\in\mathcal N}|V(v)|[-v,v].
\end{equation}
\end{corollary}

\begin{remark}
\begin{itemize}
\item[$(i)$] In Section \ref{subsec:lowerdim}, in order to simplify the analysis in the case when $\N$ does not span $\Z^d$, we also introduce the following convergence, which has been widely used in the literature: given $X\subset \L$ we define the \textit{empirical measure}
\begin{equation}\label{eq:empiricaldef}
\mu_N(X):=\frac{1}{N}\sum_{x\in X}\delta_{x/N^{1/d}}.
\end{equation}
Then by definition $X_N$ converges to a set $E$ if the empirical measures $\mu_N(X_N)$ converge to $\1_E$ weakly as measures. Observe that, when considering subsets of $\Z^d$, this convergence is equivalent to the “$L^1_{loc}$” considered above, and to state and prove Theorem \ref{thm:main} we could equivalently use the functionals
\begin{equation}
    \F_N(\mu):=\begin{cases}
        \F(X) & \text{if $\mu=\mu_N(X)$ for some $X$, $\# X=N$}\\
        +\infty & \text{otherwise}\end{cases}.
\end{equation}
\item[$(ii)$] The restriction $\# X_N=N$ in \eqref{rescenergy} is not necessary, and Theorem \ref{thm:main} would be true even without it. However we chose to put it so that the proof of the recovery sequence becomes more precise (we can construct sets with \textit{exactly} $N$ points), and so that we can talk about minimizers of $\F_N$ (which without a cardinality constraint would be trivial) and thus state Corollary \ref{cor:Wulff}.
\end{itemize}
\end{remark}

A particular case of Theorem \ref{thm:main} appears in \cite{AFS} within the study of triangular-lattice configurations in the plane. This global convergence result for discrete energy functionals was successively made more quantitative near the minimum in \cite{schmidt}, who proved the $N^{3/4}$-law for fluctuations near the minimizer in the hexagonal case (see also \cite{mainini2019formula} for the $3$-dimensional case and \cite{cicalese2019maximal,mainini2020maximal} more in general). Results describing the structure of configurations minimizing important discrete functionals in an “unconstrained” setting, i.e. without restricting the configurations to a lattice, are available in very few cases, in dimensions $2,3,8,24$ in different models (see \cite{theil2006proof, flatley2015face, davoli2016wulff, mapiste, viazovska2017sphere,cohn2017sphere,cohn2019universal,bedepe}), and in these cases too, the optimal limit shape can be shown to coincide with the Wulff shape of the corresponding lattice. Of the above works, note that \cite{bedepe, cohn2019universal} work with a potential $V$ which involves interactions beyond nearest-neighbors, motivating our choice of including general $V$ in practice.

\medskip 

Our proof of the $\Gamma$-liminf inequality in Theorem \ref{thm:main} is based on splitting the contributions to the energy appearing in \eqref{eq:rescenergy} into contributions from the single edge directions in the support of $V$, and taking the limit on each one separately with the help of Reshetnyak's theorem. The $\Gamma$-limsup inequality is by polyhedral approximation, like the one performed in a special case in \cite{AFS}.

\medskip

One of the advantages of our method, especially for the $\Gamma$-liminf case, is that it has indicated us a strategy for treating the quasiperiodic case, via a relatively non-technical discussion. We expect that the same strategy can extend to more general quasicrystals and glass-like generalizations to configurations constructed from configurations of hypersurfaces.

\subsection{Setting and main results in the quasicrystal case}\label{ssec:qcsetup}

The term {\bf quasicrystal} refers to a class of generalized lattices that are not periodic, but possess some form of \textit{quasiperiodicity}. A great interest in these kinds of arrangements arose in crystallography in the 80's (see e.g. \cite{kramer1984periodic,duneau1985quasiperiodic,levine1986quasicrystals}), when it was famously observed by Shechtman \cite{shechtman1984metallic} that some metal alloys create diffraction patterns with five-fold symmetries that could not be explained by periodic arrangements of atoms. These patterns were then explained exactly by a “quasiperiodic” structure that never repeats but has atomic Fourier transform (thereby indicating some version of periodicity). We refer to \cite{senechal1996quasicrystals} for a mathematical introduction to quasicrystals. In the mathematical community the most famous quasiperiodic arrangement is arguably the Penrose tiling, a tiling of the plane created with the use of two kinds of rhombuses as described by De Bruijn in \cite{de1981algebraic}. An algebraic precursor of the idea of a quasicrystal can be traced back to Meyer \cite{meyer1972algebraic}. 

\medskip 
Quasicrystals can be satisfactorily modeled by a variety of alternative non-equivalent mathematical definitions, depending on the precise focus of a given model or theory, and we refer to \cite{lagarias1999geometric1,lagarias1999geometric2, GahRhy86} for a comparison between some (but not all) of the different possible definitions.

\medskip The choice of definition which allows to directly connect to the theorems in Section \ref{ssec:persetup} is the so-called “multigrid construction” of quasicrystals, introduced as far as we could find in De Bruijn \cite{de1981algebraic} in the second part of his paper, and extended in \cite{GahRhy86} to the setting considered here. Wulff shapes of quasicrystals have been compellingly characterized in the physics literature for example in \cite{Ho+89, IngSte89}, and thus our work here consists in writing complete proofs of the energy convergence which formalizes \cite{Ho+89, IngSte89} within the theory of $\Gamma$-convergence, and slightly generalizes the results to the full multigrid setup \cite{GahRhy86}.

Amongst other constructions of large classes of quasicrystals, we mention the \textit{cut-and-project} method, also formulated in \cite{de1981algebraic} for the Penrose tiling case. G\"ahler and Rhyner \cite{GahRhy86} extended the Penrose description from De Bruijn and proved that tilings by parallelohedra can be constructed by one method if and only if they can be constructed by the other.

\subsubsection{Energy and $\Gamma$-convergence result in the quasicrystal case} 
We now define perimeter energies on the space of finite unions of tiles in quasiperiodic tilings. We consider a given quasiperiodic tiling $\mathcal T$ of $\mathbb R^d$ by parallelotopes (“tiles”), which are produced through the “multigrid construction”. This means that the polyhedral complex of the tiling $\mathcal T$ is dual to the one formed by dissecting $\mathbb R^d$ by a number of families of parallel hyperplanes. We refer to Section \ref{ssec:multigrid} for the precise definitions, and to Figures \ref{fig:test1}, and \ref{fig:test2} for a simple example. We consider the energy of a set $T\subset\R^d$ which is a union of finitely many tiles from $\mathcal T$ as the following perimeter functional:
\[
    \mathcal E(T)=\int_{\partial T} w(\nu(x))d\mathcal H^{d-1}(x),
\]
where $\nu$ is the normal to $\partial T$, and where $w$ is a nonnegative weight function which is defined on the finitely many possible directions of $\nu$. The crucial point is that this functional can be rewritten in the form 
\begin{equation}\label{entileintro}
\mathcal{E}(T)=\mathcal{E}_{W}(T)=\sum_\nu W(\nu) \ 
    \sharp\{\mbox{facets of $\partial T$ with exterior normal $\nu$}\}
\end{equation}
for some non-negative potential $W$ (with a sign convention opposite to the crystal case), and where the sum runs among all the possible directions of $\nu$.

A more detailed description is given in Section \ref{ssec:energyqc}, in which formula \eqref{entileintro} is repeated in \eqref{entile} and reexpressed in a dual space in \eqref{entile2}. This rewriting makes it possible to interpret the perimeter-type functional \eqref{entileintro} as a superposition of interaction potentials of type \eqref{eq:rescenergy}, which we know how to handle. We then define
\begin{equation}\label{eq:FNtile}
        \mathcal F_N(E):=
        \begin{cases}
            \displaystyle N^{-\frac{d-1}{d}} \mathcal E_{W}(T) & 
            \text{if $T:=N^{\frac1d}E$ is a disjoint union of $N$ tiles from $\mathcal T$},\\
            +\infty & \text{otherwise}.
        \end{cases}
    \end{equation}
    and then we have the following analogue of Theorem \ref{thm:main}.
\begin{theorem}\label{thmqcintro}
    The functionals $\mathcal F_N$ defined in \eqref{eq:FNtile} $\Gamma$-converge, 
    with respect to the $L^1_{loc}$ topology, to the functional $P_{W}=P_{\phi_{W}}$, for $\phi_{W}$ of the same form as \eqref{eq:phiV} (see \eqref{phivnew} and the preceding discussion for the definition). Moreover, if $W(\nu)>0$ for every normal $\nu$ to a tile, sequences with equibounded energy are compact in $L^1_{loc}$.
    In particular, the minimizers $\overline{X}_N$ of $\mathcal E_{W}$ from \eqref{entileintro},
    amongst $N$-tile configurations converge, 
    up to rescaling, to a finite perimeter set $E$ that minimizes the anisotropic perimeter $P_{W}$.
\end{theorem}

\subsection{The search of general Wulff shape constructions}\label{ssec:wshapeintro}
In both our main theorems, we find that the limit anisotropic perimeter functionals correspond to $\phi$ which is a sum of terms of the form $W(v)|\langle v, \cdot\rangle|$ with $W(v)>0$ (possibly take $W=-V$ for the crystal case). What are the possible Wulff shapes corresponding to these perimeters? Surprisingly, this natural question, which is thoroughly investigated in physics papers \cite{Ho+89, IngSte89}, does not seem to be well-studied in the mathematical literature. Therefore, in Section \ref{sec:wshapes} we collect and describe the basic results in this direction and provide a few new examples to illustrate some phenomena. It follows from classical convex geometry (see \cite{Sch14}) that if $\phi_{W}=\phi_{W_1}\pm\phi_{W_2}$ where $W_1,W_2$ are finitely supported potentials and $\phi_{W}$ is defined as in \eqref{eq:phiV}, then the Wulff shape $\mathcal{W}_{W}$ is the Minkowski sum/difference of the Wulff shapes of $\phi_{W_j},j=1,2$. Therefore for positive finitely supported $V$ the corresponding Wulff shape is a Minkowski sum of segments, sometimes named a \emph{zonotope}, and we directly have the following:

\begin{theorem}[Wulff shapes under constant-sign potentials]\label{prop-zonotope}
A set $\mathcal W\subset \mathbb R^d$ is obtainable as limit optimal shape from energies as in \eqref{eq:rescenergy} for nonpositive $V$ with finite support, or as in \eqref{entileintro} for positive $W$, if and only if $\mathcal W$ is a zonotope.
\end{theorem}

It is often claimed in the mathematical literature that this remains the case even for \emph{signed} finitely supported $W$. However we point out that (as observed amongst others by \cite{IngSte89} in the quasiperiodic case) this is not true. Indeed we can find examples of simple signed $W$, both coming from lattices and from multigrid quasicrystals, in which the Wulff shape is not a zonotope. This indicates that real-world crystalline shapes such as the pyritohedron or general truncated octahedra, are possible within our model, for signed $W$.

\medskip We leave as an interesting future direction the extension of $\Gamma$-convergence results such as Theorems \ref{thm:main} and \ref{thmqcintro} to the case of general signed $W$ (with $W=-V$ for Theorem \ref{thm:main}). We believe that the results remain true whenever $W$ is such that $\phi_{W}$ is strictly positive on the unit sphere. The techniques considered here need to be refined for that case and we leave the extension to future work.

\subsection{Structure of the paper}
In Section \ref{sec:reduction} we prove equation \eqref{chcoordaff} which allows us to change coordinates and reduce the study of lattice energies to the case of $\mathbb Z^d$. Section \ref{sec:proof} is devoted to the crystal case, with the proof of Theorem \ref{thm:main} and of Proposition \ref{prop:compactness}, as well as to some degenerate analogues, described in Section \ref{subsec:lowerdim}. Section \ref{sec:qc} is devoted to the quasicrystal case, with the proof of Theorem \ref{thmqc}. Section \ref{sec:wshapes} is devoted to the study of possible Wulff shapes that can appear as continuum minimizers for the limit energies from our main theorems. It also describes the first steps for the study of Wulff shapes for signed potentials $W$. Finally, Section \ref{sec:finalrem} includes sketches of some direct generalizations of our results and a short discussion of what seem interesting open directions for future work.

\paragraph{Acknowledgements.} The authors wish to thank Maria Stella Gelli for pointing out that Theorem \ref{thm:main} was a specialization of previous results \cite{BraGel, AliGel, Gel}. GDN has received funding from the European Research Council (ERC) under the European Union’s Horizon 2020 research and innovation programme under grant agreement No 757254 (SINGULARITY). MP is supported by the Chilean Fondecyt Iniciaci\'on grant number 11170264 entitled “Sharp asymptotics for large particle systems and topological singularities”.

\section{Lattice case}\label{sec:proof}
In this Section we prove Theorem \ref{thm:main}. We first prove in Subsection \ref{sec:reduction} that we can reduce to the case of $\mathcal{L}=\Z^d$. In Subsection \ref{subsec:splitting} we split the energy according to the direction of the bonds appearing in \eqref{eq:rescenergy}. We then relate each of these energies to a suitable anisotropic perimeter of a certain set associated to $X_N$. In this way we can rewrite the total energy $\E_N$ as the superposition of anisotropic perimeters in different directions $v$ (those for which $V(v)$ is non zero) of certain approximations of $X_N$ as union of cylinders with axis along $v$. This will help us deduce in Subsection \ref{latliminf} the $\liminf$ inequality from the lower semicontinuity of perimeter-type functionals. Then in Subsection \ref{sec:glimsup1} we prove the $\limsup$ inequality by approximation with polyhedral sets through a direct construction.
\subsection{Reduction to the integer lattice}\label{sec:reduction}
Every lattice $\L$ can be expressed as $\L=M\Z^d$ with $M\in GL(d)$. Moreover there is a direct correspondence between configurations in $\L$ and in $\Z^d$:
\[
\begin{tabular}{ccc}
    $\Z^d=M^{-1}\L$ & $\longleftrightarrow$ & $\L$\\
    $\widetilde \N:=M^{-1}\N$ & $\longleftrightarrow$ & $\N$\\
    $\widetilde X_N:=M^{-1}X_N$ & $\longleftrightarrow$ & $X_N$\\
    $\widetilde V:=V\circ M$& $\longleftrightarrow$  & $V$\\
    $M^{-1}E_h\overset{L^1}{\to}M^{-1}E$ & $\longleftrightarrow$ & $E_h\overset{L^1}{\to}E$\\
    $\widetilde E=M^{-1}E$ & $\longleftrightarrow$ & $E$
\end{tabular}
\]
Then $\E_{\widetilde V}(\widetilde X_N)=\E_V(X_N)$ and thus to find the $\Gamma$-limit for a general lattice $\L$ we can translate the problem in $\Z^d$, find the $\Gamma$-limit there, and then go back to the original lattice. The following result shows that when translating the problem from $\Z^d$ to any lattice $\L$, the perimeter functional $P_V$ in \eqref{eq:defper} behaves well.

\begin{proposition}[Equivariance under linear mappings]\label{prop:equivariance} Given $M\in GL(d)$ with $\det M>0$ and $P_V$ as defined by \eqref{eq:defper} and \eqref{eq:phiV}, we have
\[
\frac{1}{\det M}P_V(E)=P_{V\circ M} (M^{-1}E).
\]
\end{proposition}

\begin{proof}
We apply the area formula \cite[Thm.~2.91]{AFP} to the map $M^{-1}$ and the $(n-1)$-rectifiable set $\partial^* E$:
\begin{equation}\label{eq:equiv1}
    \int_{\partial^* \widetilde E} \phi_{\widetilde V}(\nu_{\widetilde E}(x))d\H^{d-1}(x)= \int_{\partial^* E} \phi_{\widetilde V}(\nu_{\widetilde E}(M^{-1}y)) (J^{\nu_{ E}(y)^\perp} M^{-1}) \,d\H^{d-1}(y),
\end{equation}
where $J^{\nu_{ E}(y)^\perp} M^{-1}$ is the Jacobian determinant of $M^{-1}$ restricted to the hyperplane orthogonal to $\nu_E(y)$.

\medskip

Now we claim that $\nu_{\widetilde E}(M^{-1}y)=\frac{M^*\nu_{E}(y)}{|M^*\nu_{E}(y)|}$. Indeed, choose a basis $e_1,\ldots, e_{d-1}$ for the tangent space to $\widetilde E$, which is equal to $\nu_{\widetilde E}^\perp$. Then $Me_1,\ldots, Me_{d-1}$ is a basis for the tangent space to $E$, which coincides with $\nu_{E}^\perp$. Therefore we have $0=\langle \nu_{E},Me_i\rangle=\langle M^*\nu_E,e_i\rangle$. As $M^*\nu_E$ is orthogonal to $e_1,\ldots,e_{d-1}$, it must be a multiple of $\nu_{\widetilde E}(M^{-1}y)$, as desired. We thus obtain that \eqref{eq:equiv1} equals
\[
\int_{\partial^* E} \phi_{\widetilde V}\left(\frac{M^*\nu_{ E}(y)}{|M^*\nu_{ E}(y)|}\right) (J^{\nu_{ E}(y)^\perp} M^{-1}) \,d\H^1(y).
\]
We have
\begin{align*}
    \phi_{\widetilde V}(M^*\nu_E)&=\sum_{\widetilde w\in \widetilde\N}|\widetilde V(\widetilde w)|\langle M^*\nu_E, \widetilde w\rangle_+=\sum_{\widetilde w\in \widetilde \N}|V(M\widetilde w)| \langle\nu_E, M\widetilde w\rangle_+\\
    &=\sum_{ w\in \N}|V(w)|\langle\nu_E, w\rangle_+=\phi_V(\nu_E)
\end{align*}
thus
\begin{equation}\label{eq:cov}
\int_{\partial^* \widetilde E} \phi_{\widetilde V}(\nu_{\widetilde E}(x))d\H^1(x)=\int_{\partial^*  E} \phi_{V}(\nu_{E}(y))\frac{(J^{\nu_{ E}(y)^\perp}M^{-1})}{|M^*\nu_E(y)|}d\H^1(y).
\end{equation}
We now want to prove that
\begin{equation}\label{eq:detidentity0}
\frac{(J^{\nu_{ E}(y)^\perp}M^{-1})}{|M^*\nu_E(y)|}=\frac{1}{\det M}.
\end{equation}

We first claim that
\begin{equation}\label{eq:detidentity}
J^{\nu_{\widetilde E}^\perp} M=\frac{\det M}{|\pi_{\nu_E}(M\nu_{\widetilde E})|}.
\end{equation}

The tangential jacobian $J^V M$ of a linear map $M$ with respect to a hyperplane $V$ can be computed in the following way: consider an $(n-1)$-dimensional unit cube $Q$ inside $V$, then $J^VM=\H^{n-1}(MQ)$. Consider now a cube expressed as a Minkowski sum $Q'=Q + [0,\nu_V]$, where $\nu_V$ is the normal to $V$. Then $\det M=\H^n(MQ')=|\pi_{(MV)^\perp}M\nu_V| \, \H^{n-1}(MQ)$. From this \eqref{eq:detidentity} follows as a special case with $V=\nu_{\widetilde E}^\perp$.

Now we prove identity \eqref{eq:detidentity0}. 
First of all we compute $|\pi_{\nu_E}(M\nu_{\widetilde E})|$. We use the fact that $\nu_{\widetilde E}(M^{-1}y)=\frac{M^*\nu_E(y)}{|M^*\nu_E(y)|}$ and the fact that $\langle M\nu_{\widetilde E}, (M^*)^{-1}\nu_{\widetilde E}\rangle = \langle \nu_{\widetilde E},M^*(M^*)^{-1}\nu_{\widetilde E} \rangle=1$:
\begin{align*}
    \pi_{\nu_E}(M\nu_{\widetilde E})&=\big\langle M\nu_{\widetilde E}, \nu_E \big\rangle \nu_E\\
    &=\left\langle M\nu_{\widetilde E}, \frac{(M^*)^{-1}\nu_{\widetilde E}}{|(M^*)^{-1}\nu_{\widetilde E}|} \right\rangle\frac{(M^*)^{-1}\nu_{\widetilde E}}{|(M^*)^{-1}\nu_{\widetilde E}|}\\
    &=\frac{(M^*)^{-1}\nu_{\widetilde E}}{|(M^*)^{-1}\nu_{\widetilde E}|^2}
\end{align*}
and thus $|\pi_{\nu_E}(M\nu_{\widetilde E})|=\frac{1}{|(M^*)^{-1}\nu_{\widetilde E}|}$. Therefore, using also the fact that $\nu_{\widetilde E}=\frac{M^*\nu_E}{|M^*\nu_E|}$ implies $|(M^*)^{-1}\nu_{\widetilde E}||M^*\nu_E|=1$, we get:
\[
J^{\nu_{\widetilde E}^\perp} M=\frac{\det M}{|\pi_{\nu_E}(M\nu_{\widetilde E})|}=\det M \,|(M^*)^{-1}\nu_{\widetilde E}|=\frac{\det M}{|M^*\nu_E|}.
\]
Due to the fact that restriction to a subspace and inverse commute for invertible maps, we also find
\[
(J^{\nu_{ E}(y)^\perp}M^{-1})=(J^{\nu_{\widetilde E}^\perp} M)^{-1}
\]
and therefore
\[
\frac{(J^{\nu_{ E}(y)^\perp}M^{-1})}{|M^*\nu_E(y)|}=\frac{1}{J^{\nu_{\widetilde E}^\perp} M \,|M^*\nu_E|}=\frac{1}{\det M}.
\]
This proves the claimed identity \eqref{eq:detidentity0}.
Equation \eqref{eq:cov} thus becomes
\begin{equation}\label{chcoordaff}
    \int_{\partial^* \widetilde E} \phi_{\widetilde V}(\nu_{\widetilde E}(x))d\H^{d-1}(x)=\frac{1}{\det M}\int_{\partial^*  E} \phi_{V}(\nu_{E}(y))d\H^{d-1}(y).
\end{equation}

This concludes the proof.
\end{proof}

A direct consequence of the previous result is that we can reduce to the case of $\Z^d$, and if we prove the $\Gamma$-convergence on $\Z^d$ we automatically prove it for every lattice.

\begin{remark} 
A version of \eqref{chcoordaff} can be proved via Minkowski sums if $\phi_V$ is a positively $1$-homogeneous convex functional, in a more straightforward way (see for example \cite[end~of Proof~of~Thm.~1.1]{FMP}). The above computations are motivated by the wish to be able to consider more general functions $\phi_V$, for which the only requirement is to be positively $1$-homogeneous, in future works.
\end{remark}

\subsection{Splitting of the energy with respect to sublattices}\label{subsec:splitting}

\subsubsection{Splitting of the lattices}
Given a vector $v\in \Z^d$, we consider the set
\[
\L_{v^\perp}=\{w\in \Z^d: \langle w, v\rangle=0\}.
\]
This is a sublattice of $\Z^d$ of dimension $d-1$ (see e.g. \cite[Ch. VII.2]{bourbaki1998general}), lying in the linear subspace $(\R v)^\perp$. As a lattice of such subspace, we can select a basis $b_1,\ldots, b_{d-1}$ and define its fundamental cell
\[
U_{\L_{v^\perp}}:=\left\{\sum_{i=1}^{d-1} t_ib_i: t_i\in[0,1)\quad\forall i=1,\ldots,d-1\right\}.
\]

We finally consider the sublattice 
\[
\L_v:=\L_{v^\perp}\oplus \mathrm{span}_{\Z}v.
\]
and its fundamental cell
\[
U_v:=U_{\L_{v^\perp}}+[0,v)=\left\{t_d v+\sum_{i=1}^{d-1} t_ib_i: t_i\in[0,1)\quad\forall i=1,\ldots,d\right\}.
\]
Given a translation vector $\tau\in \Z^d\cap U_v$ we also define the translated sublattices $\L^{v,\tau}:=\tau+\L_v$. For fixed $v$, a volume estimate using the fundamental cells gives that the number of distinct sublattices of the form $\L^{v,\tau}$ is given by
\[
\# (\Z^d\cap U_v)= \H^d(U_v)=\H^{d-1}(U_{\L_{v^\perp}})|v|.
\]

\subsubsection{Splitting of the energy}\label{sec-spliten}
For every fixed $v\in \Z^d\setminus \{0\}$ and every fixed translation vector $\tau\in\Z^d$ we define an energy functional by setting for $X\subset\L$
\begin{equation}\label{eq:energysplit}
\F^{v,\tau}(X):=-\sum_{x\in X\cap \L^{v,\tau}} \sum_{x'\in \L^{v,\tau}\setminus X} V(x'-x),
\end{equation}
which yields a decomposition of the total energy as
\[
\F(X)=\sum_{v\in \Z^d\setminus\{0\}} \sum_{\tau \in \Z^d\cap U_v}\F^{v,\tau}(X).
\]
We observe that for every fixed $v$ and $\tau$, the only terms appearing in $\F^{v,\tau}$ are of the type $V(v)$, so that effectively
\[
\F^{v,\tau}(X)=-V(v) \#\big\{(x,x'):x\in \L^{v,\tau}\cap X ,\, x'\in \L^{v,\tau}\setminus X\big\}.
\]

\subsubsection{Relation with anisotropic perimeter}\label{ssec:anisoper}
Now for every given $v$ and $\tau$ we associate to $X$ a set $E^{v,\tau}(X)$, made of translated copies of $U_v$, and we relate the energy $\F^{v,\tau}(X)$ to a suitably defined anisotropic perimeter in direction $v$ of $E^{v,\tau}$.

More precisely, to every configuration $X\subset \Z^d$, and to every fixed $v\in \Z^d\setminus\{0\}$ and $\tau \in \Z^d$, we associate the following set and its contraction by a factor of $N^{\frac1d}$:
\begin{equation}\label{evtau}
E^{v,\tau}(X):=\bigcup_{x\in \L^{v,\tau}\cap X}\left( x+U_v\right),\quad\mbox{and}\quad E_N^{v,\tau}(X):=N^{-\frac1d}E^{v,\tau}(X).
\end{equation}
Every paralelotope $x+U_v$ in the above definition of $E_N^{v,\tau}(X)$ has exactly one face for which normal vector $\nu$ there holds $\langle v, \nu\rangle_+> 0$, and $v=|v|\nu$ for this face. Furthermore, such face has area $\H^{d-1}(U_{\L_{v^\perp}})$ and all other faces are parallel to $v$. Hence the contribution in \eqref{eq:energysplit} coming from bonds contained in $\L^{v,\tau}$ can be expressed as an anisotropic area functional in direction $v$, as $\F^{v,\tau}(X)=P^v(E^{v,\tau}(X))$, where for finite-perimeter $E\subset R^d$ we define $P^v(E)$ as follows, $\nu(x)$ being the normal to the essential boundary $\partial^* E$ (cf. \cite[Definition~3.60]{AFP}):
\begin{equation}\label{eq:Pv} P^v(E):=\int\limits_{\partial^* E}\frac{-V(v)}{\H^{d-1}(U_{\L_{v^\perp}})}\left\langle\nu(x), \frac{v}{|v|}\right\rangle_+d\H^{d-1}(x).
\end{equation}
For a general measurable $E\subset\R^d$ and $v,\tau$ as above we define:
\begin{equation}
\F_N^{v,\tau}(E):=\left\{
    \begin{array}{ll}
        \F^{v,\tau}(X_N)&\mbox{ if }\exists X_N\subset\L,\quad \sharp X_N=N, E_N^{v,\tau}(X_N)=E,\\[3mm]
        +\infty&\mbox{ else.}
    \end{array}
\right.
\end{equation}
This functional automatically satisfies $\F_N(E)=\sum_{v\in \Z^d\setminus\{0\}} \sum_{\tau \in \Z^d\cap U_v}\F_N^{v,\tau}(E)$ for all measurable $E\subset \mathbb R$ (recall \eqref{eq:rescenergy}), and will be used in the proof of $\Gamma$-convergence.

\subsection{Liminf inequality}\label{latliminf}
We now put together two facts:
\begin{itemize}
    \item By Reshetnyak's Theorem \cite[Thm.~2.38]{AFP}, functionals as in \eqref{eq:Pv} are lower semicontinuous  with respect to the $L^1_{loc}$ convergence of sets, since the $1$-homogeneous extension of the integrand from \eqref{eq:Pv} is a convex function.
    \item If $E_N\to E$ in $L^1_{loc}$, where $E_N=E_N(X_N)$ for suitable $X_N\subset \L$ with $\sharp X_N=N$ (see definition \eqref{eq:ENdef}), then for every $v\in \Z^d\setminus\{0\}$ and $\tau\in \Z^d$, due to definition \eqref{evtau} we have $E_N^{v,\tau}={N^{-\frac{1}{d}}}E^{v,\tau}\to E$ in $L^1_{loc}$ as well.
\end{itemize}
From these two facts, given any sequence $X_N$ with $N=\sharp X_N\to\infty$ such that $E_N(X_N)\to E$ in $L^1_{loc}$ we obtain that (with $P^v$ defined in \eqref{eq:Pv})
\begin{align*}
P_V(E)&=\sum_{v\in\Z^d\setminus\{0\}}\H^{d-1}(U_{\L_{v^\perp}})|v|P^v(E)=\sum_{v\in\Z^d\setminus\{0\}} \sum_{\tau\in \Z^d\cap U_v} P^v(E)\\
&\leq \liminf_{N\to\infty}\sum_{v\in\Z^d\setminus\{0\}} \sum_{\tau\in \Z^d\cap U_v} P^v(E_N^{v,\tau}(X_N))\\
&=\liminf_{N\to\infty} N^{-\frac{d-1}{d}}\sum_{v\in\Z^d\setminus\{0\}}\sum_{\tau\in \Z^d\cap U_v} P^v(E^{v,\tau}(X))\\
& =\liminf_{N\to\infty} N^{-\frac{d-1}{d}}\sum_{v\in\Z^d\setminus\{0\}}\sum_{\tau\in \Z^d\cap U_v} \F_N^{v,\tau}(E_N(X_N))\\
&=\liminf_{N\to\infty} N^{-\frac{d-1}{d}}\F_N(E_N(X_N)).
\end{align*}
This proves the $\liminf$ inequality.

\subsection{Limsup inequality}\label{sec:glimsup1}

\begin{lemma}\label{lem-plane}
Let $\nu\in \mathbb S^{d-1}$ and let $P$ be a polytope in $\nu^\perp$ and $v\in \mathcal N, \tau \in U_v$ fixed. Then the number of edges of the form $[x,x+v), x\in\L^{v,\tau}$ that cross the subset $\mathrm{Int}[P,\L_{v,\tau}]\subset P$ given by
\begin{multline}\label{intperp}
\mathrm{Int}[P,\L_{v,\tau}]:=\\
\bigcup\left\{((U_v^\perp+\tau+x)+\R v)\cap P:\ x\in \L_{v^\perp}, ((U_v^\perp+\tau+x)+\R v)\cap \partial P=\emptyset\right\}
\end{multline}
in such a way that they have positive scalar product with $\nu$, is equal to 
\begin{equation}\label{biperp}
    \frac{\left\langle \nu,\frac{v}{|v|}\right\rangle_+}{\mathcal H^{d-1}(U_{v^\perp})}\mathcal H^{d-1}(\mathrm{Int}[P,\L_{v^\perp}]).
\end{equation}
Furthermore, we have 
\begin{equation}\label{arebound}
\mathcal H^{d-1}(P\setminus \mathrm{Int}[P,\L^{v,\tau}])\le C_d\frac{\mathrm{diam}(U_{v^\perp})}{\langle \nu, v/|v|\rangle}\mathcal H^{d-2}(\partial P)
\end{equation}.
\end{lemma}
\begin{proof}
We start by proving \eqref{biperp}. If $\langle \nu,v\rangle\le 0$ then no bonds $[x,x+v),x\in \L^{v,\tau}$ cut the hyperplane $\nu^\perp$ in a direction making positive scalar product with $\nu$, and also \eqref{biperp} gives zero contribution, thus the thesis follows. From now on we concentrate on the case $\langle \nu, v\rangle>0$. In this case an edge $[x,x+v)$ as above intersects $\nu^\perp$ if and only if it intersects it while having positive scalar product with $\nu$.

\medskip

We pass to a model case for clarity first: There exists an invertible affine map $x\mapsto Ax-\tau$ which sends $\L^{v,\tau}$ to $\Z^d$ and sends the $v+\tau$ to $e_1$. The union of all edges of the form $[x,x+v), x\in \L^{v,\tau}$ is sent by this map to the countable union of lines $\R\times\Z^{d-1}$ and the hyperplane $\nu^\perp$ is sent to a hyperplane transverse to all these lines since $\nu^\perp\nparallel v$. The hyperplane $A\nu^\perp -\tau$ then meets each line $\R\times \Z^{d-1}$ exactly once. Coming back to the original coordinates, we find that the number of intersections of a set in $\nu^\perp$ with segments $[x,x+v),x\in\L^{v,\tau}$ is equal to the number of intersections with the lines $\L_{v^\perp}+v\R +\tau$. In the case of the set $\mathrm{Int}[P,\mathcal L^{v,\tau}]$ this number is also equal to the number of cylinder sets $(U_v^\perp+\tau+x)+\R v$ appearing in the union \eqref{intperp}, since each such cylinder set contains exactly one such line.  We now claim that for each $y\in\R^d$ there holds
\begin{equation}\label{jacob}
\mathcal H^{d-1}(((U_v^\perp+y)+\R v)\cap \nu^\perp)=\frac{\mathcal H^{d-1}(U_{v^\perp})}{\left\langle \nu,\frac{v}{|v|}\right\rangle_+}.
\end{equation}
We can apply the above for $y=x+\tau$ corresponding to the ters in \eqref{intperp} to obtain \eqref{biperp} by $(d-1)$-dimensional volume comparison. To prove \eqref{jacob}, we proceed by an elementary reasoning, although faster alternative proofs are possible. Note that it suffices to prove the case $y=0$. Then observe that $\langle \nu, v\rangle_+=\langle \nu,v\rangle$ in our case. If we cut the infinite cylinder $U_v^\perp +\R v$ into equal parallelotopes by hyperplanes orthogonal to $\nu$ passing through the equally spaced points $\Z \frac{v}{|v|}$, then the average number of parallelotopes per unit length in the direction $v/|v|$ is the same as if we had cut it by hyperplanes orthogonal to $v$, thus the cut-out parallelotopes have equal volumes in the two cases. Since the volume of a parallelotope is equal to its basis area times the height relative to that basis, we get \eqref{jacob} directly.

\medskip

We now prove \eqref{arebound}. To do this, we note that any cylinder $U_{v^\perp}+y +\R v$ that meets $\partial P$ must also be included in its $R$-neighborhood for $R$ larger or equal to the diameter of the intersection of $U_{v^\perp}+\R v$ with $\nu^\perp$. The latter is bounded by $R:=\frac{\mathrm{diam}(U_{v^\perp})}{\langle \nu, v/|v|\rangle}$ by the same reasoning as in the first part of the proof. Then as $P$ is a convex polytope, we can use the Hausdorff measure of $\partial P$ to control the volume of its $R$-neighborhood.
\end{proof}

We now pass to the construction of the recovery sequence. For a given measurable set $E$ we need to construct sets $E_N=E_N(X_N)$, with $X_N\subset \Z^d, \sharp X_N=N$, such that $N^{\frac{d-1}{d}}\F_N(E_N)\to P_V(E)$.
\medskip

{\bf We first prove the statement under the further assumption that $\sharp \mathcal N<\infty$.}

\medskip

{\bf Step 1.} {\it Approximating $E$ by polyhedral sets of volume $1$.} By a classical approximation result we approximate $E$ by polyhedral sets, more precisely we find a sequence of polyhedral sets $E_j$ such that
\[
E_j\to E\quad\text{in $L^1$}\qquad\text{and}\qquad P(E_j)\to P(E)
\]
where $P$ is the standard perimeter. 
By Reshetnyak's theorem \cite[Theorem 2.38]{AFP} the same convergence holds for any anisotropic perimeter functional such as $P_V$.
Dilating by a factor converging to $1$ we can also impose that $|E_j|=1$.
\medskip

From now $E$ is assumed to be a fixed polyhedral set of volume $1$ and we work with its rescaling $N^{\frac{1}{d}}E$ which has volume $N$.
\medskip

{\bf Step 2.} {\it Approximation of polyhedral sets by discrete sets $Y_N$.}  For large $N\in\mathbb{N}$ we will consider the discrete sets
\[
Y_N:=(N^{\frac{1}{d}}E)\cap \Z^d.
\]
We have that 
\begin{equation}\label{xeps}
\left|\sharp Y_N - N\right|\le C_d \mathcal H^{d-1}(\partial E)\ N^{\frac{d-1}{d}}\quad\mbox{for}\quad N\ge N_E,
\end{equation}
in which $N_E$ only depends on $E$ (precisely, it depends on the rate of convergence of the limit in the definition of $\mathcal H^{d-1}(\partial E)$), and $C_d$ only depends on $d$. The above bound can be obtained by comparing the total volume of the $\Z^d$-translates of $[0,1)^d$ which are completely contained in $N^{\frac{1}{d}}E$ and the volume of the smallest union of such translates which contains $N^{\frac{1}{d}}E$, with $N$, which can be interpreted as the $d$-dimensional volume of $N^{\frac{1}{d}}E$.

We can then subtract or add a set of at most $k_N:=C_d \mathcal H^{d-1}(\partial E)N^{\frac{d-1}{d}}$ points to $(N^{\frac1d}E)\cap\Z^d$ in order to get a new set $X_N$ of precisely $N$ points, $\sharp X_N=N$. If $N\ge N_E$ large enough, then it is possible to organize these $k_N$ points as a cluster of scale $k_N^{\frac1d}$, which then has interface set of cardinality $C_d k_N^{\frac{d-1}{d}}$ at most. We then have for $N\ge N_E$
\begin{equation}\label{xnenrgy}
\left|\mathcal F(X_N) - \mathcal F(Y_N)\right|\le C'_d \mathcal H^{d-1}(\partial E) N^{\frac{(d-1)^2}{d^2}}\|V\|_{\ell_1(\Z^d)},
\end{equation}
where $\|V\|_{\ell_1(\Z^d)}:=\sum_{v\in\mathcal N}|V(v)|$.

\medskip

From now on we focus on approximating the value of $\mathcal F(Y_N)$, and we will use the notations of Section \ref{sec-spliten}. If we show that $\F^{v,\tau}(Y_N)$ satisfies the good bounds for all $v\in\mathcal N, \tau\in U_v$, then we can sum the bounds and use the triangle inequality for approximating $\mathcal F(Y_N)$.

\medskip

{\bf Step 3.} {\it Approximating $\F^{v,\tau}(Y_N)$ by the contribution of interiors of faces.}  We fix $v\in\mathcal N, \tau\in U_v$. We denote by $P$ a face of the polytope $N^{\frac1d}E$. We consider the discretization of $P$ adapted to $\F^{v,\tau}(Y_N)$ given by $\mathrm{Int}[P,\L^{v,\tau}]$ from \eqref{intperp}. For each face $P$ of $N^{\frac1d}E$, the total number of bonds congruent to $v$ cut by the interior of $P$ in $\L_{v^\perp}$ is controlled via Lemma \ref{lem-plane} and is given by \eqref{biperp}.

\medskip

As $P$ meets each of the parallel lines $\tau +x+\R v, x\in U_v^\perp$ at most once, the number of bonds cut by $P$ is the same as the number of lines of this type cut by $P$. Also, note that each cylindrical set of the form 
\[
C[v,\tau,x]:=U_{v^\perp} + x+\tau+\R v, \quad x\in \L^{v,\tau}
\]
contains exactly one of the above lines, and it is not counted within the contributions of $\mathrm{Int}[P,\L^{v,\tau}]$ only if $C[v,\tau, x]\cap \partial P\neq \emptyset$. The latter condition is equivalent to $\pi_{v^\perp}C[v,\tau,x]\cap \pi_{v^\perp}\partial P\neq \emptyset$. We bound from above the number of $x\in\L^{v,\tau}$ for which this happens as follows:
\begin{eqnarray}
\lefteqn{\sharp\left\{x\in \L_{v^\perp}:\ C[v,\tau,x]\cap \partial P\neq\emptyset\right\}}\nonumber\\
&\le&\frac{\mathcal H^{d-1}\left(\left\{x\in v^\perp:\ \mathrm{dist}(x,(\pi_{v^\perp}\partial P))\le \mathrm{diam}(U_{v^\perp})\right\}\right)}{\mathcal H^{d-1}(U_{v^\perp})}\nonumber\\
&\le& \frac{C_d\ \mathrm{diam}(U_{v^\perp})\mathcal H^{d-2}(\pi_{v^\perp}\partial P)}{\mathcal H^{d-1}(U_{v^\perp})}\nonumber\\
&\le& C'_d\frac{\mathrm{diam}(U_{v^\perp})^2}{\mathcal H^{d-1}(U_{v^\perp})}\mathcal H^{d-2}(\partial P),\label{bdminko}
\end{eqnarray}
where in the last step we used the fact that projections decrease Hausdorff measure.

\medskip

Summing up our reasoning so far, we show that $-\F^{v,\tau}(Y_N)$ can be approximated by the sum of one $V(v)$ for each one of the lines that meet the interiors of all faces $P$, and the error terms are bounded with the help of \eqref{bdminko} and \eqref{biperp}. Note that, in order to obtain the precise contribution in $\F^{v,\tau}$, we multiply all terms by the negative factor $V(v)$, which reverses the inequalities. We get:
\begin{multline}\label{bdint1}
0\ge V(v)\sum_{P\mbox{ face of }N^{\frac1d}E}\frac{\left\langle \nu,\frac{v}{|v|}\right\rangle_+}{\mathcal H^{d-1}(U_{v^\perp})}\mathcal H^{d-1}(\mathrm{Int}[P,\L_{v^\perp}]) + \F^{v,\tau}(Y_N)
\\\ge V(v)\sum_{P\mbox{ face of }N^{\frac1d}E} C'_d\frac{\mathrm{diam}(U_{v^\perp})^2}{\mathcal H^{d-1}(U_{v^\perp})}\mathcal H^{d-2}(\partial P).
\end{multline}
We use now \eqref{arebound} and the definition of $P^v$ from Section \ref{ssec:anisoper}, to get that
\begin{multline}\label{bdint2}
0\ge P^v(N^{\frac1d}E) +V(v)\sum_{P\mbox{ face of }N^{\frac1d}E}\frac{\left\langle \nu,\frac{v}{|v|}\right\rangle_+}{\mathcal H^{d-1}(U_{v^\perp})}\mathcal H^{d-1}(\mathrm{Int}[P,\L_{v^\perp}])\\
\ge C_d V(v)\sum_{P\mbox{ face of }N^{\frac1d}E}\frac{\mathrm{diam}(U_{v^\perp})}{\langle \nu, v/|v|\rangle}\mathcal H^{d-2}(\partial P).
\end{multline}
{\bf Step 4.} {\it Conclusion of the proof.} We may re-express the sums in \eqref{bdint1}, \eqref{bdint2} in terms of faces of $E$ itself, using the scaling properties of the functionals $\mathcal H^{d-1}, \mathcal H^{d-2}$. The sums involving $\partial P$ scale with $N^{\frac{d-2}{d}}$ whereas the other terms scale with $N^{\frac{d-1}{d}}$. Then we find, using also the previous estimate \eqref{xnenrgy}, 
\begin{align}
    0& \ge  \sum_vP^v(E)-N^{-\frac{d-1}d} \mathcal F_N(E_N(X_N))\nonumber\\
    &\ge  \sum_vP^v(E) -N^{-\frac{d-1}d} \mathcal F(Y_N) + C'_d \mathcal H^{d-1}(\partial E) N^{-\frac{d-1}{d^2}}\|V\|_{\ell_1,per}\nonumber\\
    &\ge  C'_d \mathcal H^{d-1}(\partial E) N^{-\frac{d-1}{d^2}}\|V\|_{\ell_1,per} \nonumber\\
    &\hspace{0.5cm}- C''_dN^{-\frac{1}{d}}\sum_v V(v)\frac{\sharp(U_v\cap \Z^d)\left(\mathrm{diam}(U_{v^\perp}) + \mathrm{diam}(U_{v^\perp})^2\right)}{\mathcal H^{d-1}(U_{v^\perp})}\hspace{-0.5cm}\sum_{P\mbox{ face of }E}\hspace{-0.5cm} \mathcal H^{d-2}(\partial P).\label{bdint3}
\end{align}
As the factors involving $V, E$ are independent of $N$ and $\sum_vP^v(E)=P_V(E)$, it follows that $N^{-\frac{d-1}{d}}\F_N(E_N(X_N))\to P_V(E)$, as desired.
\medskip

\subsection{Compactness}
To prove compactness we use that the energy $N^{-\frac{d-1}{d}}\F_N (X_N)$ is comparable with the perimeter of the associated sets $E_N(X_N)$ defined in \eqref{eq:ENdef}, and apply standard compactness results for finite perimeter sets. We thus obtain compactness in the topology of local convergence in measure, which is the natural one to expect (there are sequences with equibounded energy that loose mass at infinity).
\begin{remark}In the two-dimensional case, a connectedness assumption for the sequence of sets is sufficient to at least imply compactness up to translations, because perimeter controls diameter for connected sets. This is the requirement considered for instance in \cite{AFS}. However in higher dimension this is not sufficient anymore, as is clear by considering a set with a long “tentacle”.
\end{remark}
We prove compactness in the case when $\mathrm{span}_\Z\, \N=\Z^d$. For a discussion about what happens if the span is not $\Z^d$ see Subsection \ref{subsec:lowerdim}.

In order to control the perimeter of $E_N(X_N)$ with the energy $\F(X_N)$ we need the following combinatorial lemma.
\begin{lemma}\label{lemma:combinatorial}

    Consider a subset $\N\subset\Z^d$ such that $\mathrm{span}_\Z\, \N=\Z^d$, and suppose that $V(v)<0$ for every $v\in \N$. Then, recalling \eqref{eq:rescenergy} and \eqref{eq:ENdef}, we have
    \[
        P(E_N(X_N))\leq K \F(X_N)
    \]
    for some constant $K$ depending on $V$, $\N$ and $d$ only.
\end{lemma}

\begin{proof}
Since $E_N(X_N)$ is a union of cubes, the contributions to its perimeter come from $(d-1)$-dimensional faces each of which is orthogonal to some canonical basis vector $e_i$, and separate a point $x\in X_N$ and a point $x\pm e_i\not\in X_N$, for some basis vector $e_i$. Therefore
\[
P(E_N(X_N))=2\sum_{i=1}^d \#\big((X_N+e_i)\setminus X_N\big).
\]
By the assumption that $\mathrm{span}_{\mathbb Z}\,\mathcal N=\mathbb Z^d$, we can choose a basis $\mathbf{b}=(v_1,\ldots,v_d)$ of $\Z^d$ made of vectors in $\N$. This means that for every standard basis vector $e_i$ we can write
$
e_i=\sum_{j=1}^{d} a_i^j v_j
$
for some integer coefficients $a_i^j$.  Repeatedly using the rule 
\[
\#\big( (X+a+b)\setminus X\big)\leq \#\big((X+b)\setminus X\big)+\#\big((X+a)\setminus X\big)
\]
we obtain
\[
\#\big((X_N+e_i)\setminus X_N\big)\leq \sum_{j=1}^d a_i^j\#\big((X_N+v_j)\setminus X_N\big)
\leq A \sum_{j=1}^d \#\big((X_N+v_j)\setminus X_N\big)
\]
where $A=\sup_{i,j=1,\ldots,d} |a_i^j|$. Moreover
\[
\F(X_N)\geq \sum_{j=1}^d -V(v_j) \#\big((X_N+v_j)\setminus X_N\big)\geq  c\sum_{j=1}^d \#\big((X_N+v_j)\setminus X_N\big)
\]
where $c=\inf_{j=1,\ldots,d} (-V(v_j))$.
Putting everything together the conclusion follows with $K=2dA/c$.
\end{proof}

\begin{proof}[Proof of Proposition \ref{prop:compactness}]
From Lemma \ref{lemma:combinatorial} we have the bound 
\[
P(E_N(X_N))\leq K N^{-\frac{d-1}{d}}\F_N(X_N)\leq CK
\]
for some constant $K$ not depending on $N$. The claimed compactness now follows from the standard compactness of finite perimeter sets, see e.g. \cite[Theorem~3.39]{AFP}.
\end{proof}
\subsection{Sparse or lower dimensional lattices}\label{subsec:sparselow}
In this section we consider the case that a canonical ambient lattice $\mathcal L\subset \mathbb R^d$ is given, but $\mathcal N:=\{v\in\mathbb R^d:\ V(v)\neq 0\}$ does not span the whole of $\mathcal L$. As usual, we restrict to the case $\mathcal L=\mathbb Z^d$, as our problem is affine-invariant, and we consider the case $\mathrm{span}_{\mathbb Z}\,\mathcal N\neq \mathbb Z^d$. The discussion of the $\Gamma$-limit of our perimeter functional follows analogous principles in these cases, but the rescalings that we use depend on $\mathrm{dim}\,\mathrm{span}_{\mathbb R}\,\mathcal N$ as well. We start by dealing with the case that this dimension is $<d$.
\subsubsection{Lower dimensional structures}\label{subsec:lowerdim}
In the case where $W:=\mathrm{span}_\R\, \N$ is a proper subspace of $\R^d$ (say $\dim W=k<d$) the $\Gamma$-convergence result is still true, but the $L^1_{loc}$ convergence is not the natural one anymore. This is clear considering a potential with $V(\pm e_1)=-1$ and zero otherwise, where the optimal structures are $1$-dimensional segments in direction $e_1$, and therefore under the scaling \eqref{eq:ENdef} the sets $E_N(X)$ are not compact (they converge weakly to zero when seen as measures as in \eqref{eq:empiricaldef}). However it is clear that in this case, after a rescaling by a factor $N^{-1}$, the optimal structures converge weakly to segments. As we shall now sketch, there is a natural topology where compactness holds even when $\dim W<d$: the topology of “$L^1_{loc}$-convergence on every slice parallel to $W$”. Actually for simplicity we will consider the empirical measures associated to $X_N$ and use the weak convergence of measures.

We first partition $\Z^d$ according to subspaces parallel to $W$: we consider
\[
\tau(W):=\{\tau\in\R^d: (\tau+W)\cap \Z^d\neq \emptyset\}
\]
and the family
\[
\mathcal{P}(W):=\{\tau+W:\tau\in \tau(W)\}
\]
of $k$-planes parallel to $W$. 
Let us define the anisotropic dilations
\[
T_\lambda^W(x):=\lambda\pi_{W}x+\pi_{W^\perp}x.
\]
It is clear that $T_\lambda^W$ leaves $\mathcal{P}(W)$ invariant. The replacement for the sets \eqref{eq:ENdef} is given by the empirical measures
\begin{equation}\label{eq:empirical}
\mu_N^W(X):=(T_{N^{-k/d}}^W)_\#\left(\frac{1}{\#X}\sum_{x\in X}\delta_x\right).
\end{equation}
The measures $\mu_N^W(X)$ are supported on $\mathcal{P}(W)$ and have total mass $1$ if $\# X=N$. It is clear that under the energy \eqref{eq:energy} the only interaction is among points living in the same subspace parallel to $W$, while different spaces do not interact.
We define the following notion of convergence on the space of empirical measures.

\begin{definition}[Convergence on slices]\label{def:sliceconvergence}
Let $W:=\mathrm{span}_\R\, \N$, $\dim \, W=k$. A sequence of $\mu_j$ of measures on $\mathcal{P}(W)$ converges to $\mu$ if for every $\tau\in\tau(W)$ the measures $\mu_j\llcorner (\tau+W)$ converge to $\mu\llcorner (\tau+W)$.
\end{definition}

We then define the space of slices with finite perimeter, given by all sets of the form 
\begin{equation}\label{eq:finperslices}
E=\bigcup_{\tau \in \mathcal{P}(W)} E_\tau,\qquad\text{with }\quad E_\tau\subset \tau+W\quad\text{ and }\quad \sum_{\tau\in\tau(W)}P(E_\tau)<\infty
\end{equation}
where we identify $E_\tau$ with a ($k$-dimensional) finite perimeter set in $W$. On this space we define the functional
\begin{equation}\label{eq:PVW}
P_V^W(E):=\sum_{Z\in \mathcal{P}(W)}P_V(E\cap Z)
\end{equation}
where $P_V$ is defined by \eqref{eq:defper} and \eqref{eq:phiV} (and where again we identify $E\cap Z$ with a finite perimeter set inside $Z$).

We then have the following result. We omit the proof since it is a direct consequence of Theorem \ref{thm:main} and Proposition \ref{prop:compactness} applied to every slice.

\begin{proposition}[Lower dimensional structures]
    Suppose that $W:=\mathrm{span}_\R\,\N$ has dimension $1\leq k<d$, and that $\mathrm{span}_{\Z}\,\N=\Z^d\cap W$. Consider the energy \eqref{eq:energy}, a sequence $X_N$ of configurations in $\Z^d$, with $\# X_N=N$, and the associated empirical measures defined by \eqref{eq:empirical}. Then we have:
    \begin{itemize}
        \item Compactness: every sequence $\mu_N$ such that $\sup_{N}\E (\mu_N^W)<\infty$ admits a subsequence that converges in the topology of Definition \ref{def:sliceconvergence} to a measure of the form $\mu=\frac{1}{\det (\mathrm{span}_{\Z}\N)}\H^k\llcorner E$, for some $E$ as in \eqref{eq:finperslices}.
        \item $\Gamma$-convergence: the functionals \eqref{rescenergy} converge to $\frac{1}{\det\L}P_V^W$ (see \eqref{eq:PVW}) under the same topology, where we identify a set $E$ as in \eqref{eq:finperslices} with the measure $\H^k\llcorner E$.
    \end{itemize}
\end{proposition}

\subsubsection{Sparse lattices} 
We finally briefly mention the case of sparse lattices, that is when $\Z^d\cap \mathrm{span}_\R\, \N\neq \mathrm{span}_\Z\,\N$ (for example $\N=\{\pm 2e_1,\ldots,\pm 2e_d\}$, $\mathrm{span}_\Z\, \N=2\Z^d$). It is clear that in this case we can represent $\Z^d=\bigcup_{x\in \mathcal C}(\mathrm{span}_\Z\N +x)$ as union of cosets in the finite quotient group $\Z^d/\mathrm{span}_\Z\N$ for a choice of coset representatives $\mathcal C\subset\R^d$.  The energy \eqref{eq:energy} of a finite $X\subset \Z^d$ splits as a sum of energies of coset subconfigurations $X\cap (\mathrm{span}_\Z\,\N+x), x\in\mathcal C$, which do not interact. In this case the statement of the $\Gamma$ convergence has to be modified to take into account the possibility that each such $X\cap (\mathrm{span}_\Z\N+x)$, when rescaled, converges to a different finite perimeter set. The limit space is thus composed of finite superpositions of finite perimeter sets, and the $\Gamma$-limit is a sum of energies of the type \eqref{eq:defper} (or \eqref{eq:PVW} in case of lower dimensional structures) among such a decomposition. We do not write the detailed results and we limit to observe here that they can be deduced in case of need with little modifications, and with the help of the observations above, from the main results.

\section{Quasicrystal case}\label{sec:qc}
The purpose of this Section is to extend the results of Section \ref{sec:proof} to nearest-neighbor interaction energies for subsets of quasicrystal tilings. 

\subsection{Multigrid construction of quasicrystals}\label{ssec:multigrid}

We here introduce the definition and notation for multigrids and for the associated dual tiling. 

\subsubsection{Multigrids in the “dual space”}\label{ssec:mgds} We start with the construction of multigrids, which we imagine to live in a space “dual” to the one in which our quasicrystal tiling will live.

\medskip

For $g\in\R^d, \gamma\in\R$ we define a \emph{hyperplane grid} as follows:
\begin{eqnarray}\label{1grid}
\mathcal H(g,\gamma)&:=&\{H(g,k,\gamma):\ k\in\Z\},\quad \mbox{where}\nonumber\\
H(g,k,\gamma)&:=&\{x\in\R^d: \langle x,\tfrac{g}{|g|}\rangle -\gamma =k|g|\}.
\end{eqnarray}
If $\mathcal G\subset\R^d$ is a finite set and to each $g\in\mathcal G$ we associate a number $\gamma_g\in\R$, we call the \emph{multigrid with normals $\mathcal G$ and translations $\gamma$} the collection of hyperplanes
\begin{equation}\label{mgrid}
    M(\mathcal G,\gamma):=\bigcup_{g\in G}\mathcal H(g,\gamma_g).
\end{equation}

We will sometimes avoid the mention of $\gamma, \mathcal G$ when they are clear from the context, and will write $\mathcal H_g:=\mathcal H(g, \gamma_g)$ and $H_{g,k}:=H(g,k,\gamma_g)$ in that case. 
With this convention, for $g\in\G$ and $k\in\mathbb{Z}$ we also define the slabs
\begin{equation}\label{mgslab}
S_{g,k}=S_{g,k,\gamma_g}:=\{x\in\R^d: \langle x, \tfrac{g}{|g|}\rangle -\gamma_g\in \left(k, k+1\right)|g|\}.
\end{equation}

We will assume in the following that every multigrid we consider is in general position, that is, we suppose that any $d$ hyperplanes $\{H_{k_g,g}:\ g\in J\}$, with 
$J\subset\mathcal G, \sharp J=d$ must intersect at a single point, which we denote by
$x(J,k_J)$, where $k_J:=(k_g)_{g\in J}$. Moreover we will assume that no more than $d$ hperplanes intersect simultaneously.

\begin{figure}
 
\end{figure}

\begin{figure}
\centering
\begin{minipage}{.45\textwidth}
  \centering
     \begin{tikzpicture}[line cap=round,line join=round,>=triangle 45,x=1.0cm,y=1.0cm,scale=0.7]
\clip(-3.4275934979980898,-12.326560822682575) rectangle (5.170058951469034,-4.493603422807421);
\draw [line width=1.pt,color=xfqqff,domain=-3.4275934979980898:5.170058951469034] plot(\x,{(--14.--2.*\x)/1.});
\draw [line width=1.pt,color=xfqqff,domain=-3.4275934979980898:5.170058951469034] plot(\x,{(--11.34843407987061--2.*\x)/1.});
\draw [line width=1.pt,color=xfqqff,domain=-3.4275934979980898:5.170058951469034] plot(\x,{(--8.696868159741223--2.*\x)/1.});
\draw [line width=1.pt,color=xfqqff,domain=-3.4275934979980898:5.170058951469034] plot(\x,{(--6.045302239611835--2.*\x)/1.});
\draw [line width=1.pt,color=xfqqff,domain=-3.4275934979980898:5.170058951469034] plot(\x,{(--3.3937363194824464--2.*\x)/1.});
\draw [line width=1.pt,color=xfqqff,domain=-3.4275934979980898:5.170058951469034] plot(\x,{(--0.742170399353058--2.*\x)/1.});
\draw [line width=1.pt,color=xfqqff,domain=-3.4275934979980898:5.170058951469034] plot(\x,{(-1.9093955207763305--2.*\x)/1.});
\draw [line width=1.pt,color=xfqqff,domain=-3.4275934979980898:5.170058951469034] plot(\x,{(-4.560961440905719--2.*\x)/1.});
\draw [line width=1.pt,color=xfqqff,domain=-3.4275934979980898:5.170058951469034] plot(\x,{(-7.212527361035107--2.*\x)/1.});
\draw [line width=3.pt,color=xfqqff,domain=-3.4275934979980898:5.170058951469034] plot(\x,{(-9.864093281164497--2.*\x)/1.});
\draw [line width=1.pt,color=qqqqff,domain=-3.4275934979980898:5.170058951469034] plot(\x,{(--1.1116006898738906--0.08108590502164947*\x)/1.487069462060853});
\draw [line width=1.pt,color=qqqqff,domain=-3.4275934979980898:5.170058951469034] plot(\x,{(-1.588984180848008--0.08108590502164947*\x)/1.487069462060853});
\draw [line width=1.pt,color=qqqqff,domain=-3.4275934979980898:5.170058951469034] plot(\x,{(-4.289569051569907--0.08108590502164947*\x)/1.487069462060853});
\draw [line width=1.pt,color=qqqqff,domain=-3.4275934979980898:5.170058951469034] plot(\x,{(-6.990153922291806--0.08108590502164947*\x)/1.487069462060853});
\draw [line width=1.pt,color=qqqqff,domain=-3.4275934979980898:5.170058951469034] plot(\x,{(-9.690738793013704--0.08108590502164947*\x)/1.487069462060853});
\draw [line width=1.pt,color=qqqqff,domain=-3.4275934979980898:5.170058951469034] plot(\x,{(-12.391323663735603--0.08108590502164947*\x)/1.487069462060853});
\draw [line width=3.pt,color=qqqqff,domain=-3.4275934979980898:5.170058951469034] plot(\x,{(-15.091908534457502--0.08108590502164947*\x)/1.487069462060853});
\draw [line width=1.pt,color=qqqqff,domain=-3.4275934979980898:5.170058951469034] plot(\x,{(-17.792493405179403--0.08108590502164947*\x)/1.487069462060853});
\draw [line width=1.pt,color=qqqqff,domain=-3.4275934979980898:5.170058951469034] plot(\x,{(-20.493078275901304--0.08108590502164947*\x)/1.487069462060853});
\draw [line width=1.pt,color=xfqqff,domain=-3.4275934979980898:5.170058951469034] plot(\x,{(-12.515659201293884--2.*\x)/1.});
\draw [line width=1.pt,color=xfqqff,domain=-3.4275934979980898:5.170058951469034] plot(\x,{(-15.167225121423272--2.*\x)/1.});
\draw [line width=1.pt,color=xfqqff,domain=-3.4275934979980898:5.170058951469034] plot(\x,{(-17.81879104155266--2.*\x)/1.});
\draw [line width=1.pt,color=xfqqff,domain=-3.4275934979980898:5.170058951469034] plot(\x,{(-20.47035696168205--2.*\x)/1.});
\draw [line width=1.pt,color=xfqqff,domain=-3.4275934979980898:5.170058951469034] plot(\x,{(-23.121922881811436--2.*\x)/1.});
\draw [line width=1.pt,color=xfqqff,domain=-3.4275934979980898:5.170058951469034] plot(\x,{(-25.773488801940825--2.*\x)/1.});
\draw [line width=1.pt,color=xfqqff,domain=-3.4275934979980898:5.170058951469034] plot(\x,{(-28.425054722070215--2.*\x)/1.});
\draw [line width=1.pt,color=ccqqqq,domain=-3.4275934979980898:5.170058951469034] plot(\x,{(--8.439988633657766--0.8695754987623143*\x)/-0.5680717338602375});
\draw [line width=1.pt,color=ccqqqq,domain=-3.4275934979980898:5.170058951469034] plot(\x,{(--7.36112159079906--0.8695754987623143*\x)/-0.5680717338602375});
\draw [line width=1.pt,color=ccqqqq,domain=-3.4275934979980898:5.170058951469034] plot(\x,{(--6.282254547940356--0.8695754987623143*\x)/-0.5680717338602375});
\draw [line width=1.pt,color=ccqqqq,domain=-3.4275934979980898:5.170058951469034] plot(\x,{(--5.203387505081651--0.8695754987623143*\x)/-0.5680717338602375});
\draw [line width=1.pt,color=ccqqqq,domain=-3.4275934979980898:5.170058951469034] plot(\x,{(--4.124520462222947--0.8695754987623143*\x)/-0.5680717338602375});
\draw [line width=1.pt,color=ccqqqq,domain=-3.4275934979980898:5.170058951469034] plot(\x,{(--3.0456534193642435--0.8695754987623143*\x)/-0.5680717338602375});
\draw [line width=1.pt,color=ccqqqq,domain=-3.4275934979980898:5.170058951469034] plot(\x,{(--1.9667863765055391--0.8695754987623143*\x)/-0.5680717338602375});
\draw [line width=1.pt,color=ccqqqq,domain=-3.4275934979980898:5.170058951469034] plot(\x,{(--0.8879193336468347--0.8695754987623143*\x)/-0.5680717338602375});
\draw [line width=1.pt,color=ccqqqq,domain=-3.4275934979980898:5.170058951469034] plot(\x,{(-0.19094770921186965--0.8695754987623143*\x)/-0.5680717338602375});
\draw [line width=1.pt,color=ccqqqq,domain=-3.4275934979980898:5.170058951469034] plot(\x,{(-1.269814752070574--0.8695754987623143*\x)/-0.5680717338602375});
\draw [line width=1.pt,color=ccqqqq,domain=-3.4275934979980898:5.170058951469034] plot(\x,{(-2.348681794929279--0.8695754987623143*\x)/-0.5680717338602375});
\draw [line width=1.pt,color=ccqqqq,domain=-3.4275934979980898:5.170058951469034] plot(\x,{(-3.4275488377879832--0.8695754987623143*\x)/-0.5680717338602375});
\draw [line width=1.pt,color=ccqqqq,domain=-3.4275934979980898:5.170058951469034] plot(\x,{(-4.506415880646688--0.8695754987623143*\x)/-0.5680717338602375});
\draw [line width=1.pt,color=ccqqqq,domain=-3.4275934979980898:5.170058951469034] plot(\x,{(-5.5852829235053925--0.8695754987623143*\x)/-0.5680717338602375});
\end{tikzpicture}
  \caption{Multigrid consisting of three families of lines in the plane, coded by colour. Two lines are highlighted.}
  \label{fig:test1}
\end{minipage}%
\hspace{5mm}
\begin{minipage}{.45\textwidth}
  \centering
  \definecolor{xfqqff}{rgb}{0.4980392156862745,0.,1.}
\definecolor{zzttff}{rgb}{0.6,0.2,1.}
\definecolor{ccqqqq}{rgb}{0.8,0.,0.}
\definecolor{qqqqff}{rgb}{0.,0.,1.}
\begin{tikzpicture}[line cap=round,line join=round,>=triangle 45,x=1.0cm,y=1.0cm,scale=0.27]
\clip(-11.271012882133075,-5.262184867529919) rectangle (8.57684902155027,15.1324500913512);
\fill[line width=2.pt,color=qqqqff,fill=qqqqff,fill opacity=0.3499999940395355] (-12.129177104210012,3.7196600026564863) -- (-9.705300397407097,1.4066497176664963) -- (-7.865054483207117,2.4630129672759318) -- (-5.441177776404202,0.1500026822859417) -- (-3.600931862204221,1.2063659318953772) -- (-1.1770551554013062,-1.1066443530946128) -- (0.6631907587986747,-0.050281103485177336) -- (3.0870674656015895,-2.363291388475167) -- (4.9273133798015705,-1.3069281388657314) -- (7.351190086604485,-3.6199384238557206) -- (9.191436000804465,-2.563575174246285) -- (11.61531270760738,-4.876585459236274) -- (13.455558621807361,-3.820222209626839) -- (13.963768054206902,-0.14337880130606706) -- (12.123522140006921,-1.1997420509155026) -- (9.699645433204008,1.1132682340744866) -- (7.859399519004027,0.0569049844650511) -- (5.435522812201112,2.3699152694550403) -- (3.595276898001131,1.3135520198456048) -- (1.1714001911982161,3.6265623048355944) -- (-0.6688457230017648,2.570199055226159) -- (-3.0927224298046796,4.883209340216149) -- (-4.93296834400466,3.8268460906067134) -- (-7.356845050807575,6.139856375596704) -- (-9.197090965007556,5.083493125987268) -- (-11.62096767181047,7.3965034109772585) -- cycle;
\fill[line width=2.pt,color=xfqqff,fill=xfqqff,fill opacity=0.3199999928474426] (7.124926005994308,17.51871240490487) -- (9.548802712797222,15.205702119914882) -- (7.708556798597241,14.149338870305446) -- (5.8683108843972605,13.09297562069601) -- (5.3601014519977195,9.416132212375238) -- (3.5198555377977385,8.359768962765802) -- (3.011646105398197,4.6829255544450294) -- (1.1714001911982161,3.6265623048355944) -- (-0.6688457230017648,2.570199055226159) -- (-1.1770551554013062,-1.1066443530946128) -- (-3.017301069601287,-2.1630076027040483) -- (-4.857546983801268,-3.219370852313484) -- (-5.36575641620081,-6.896214260634256) -- (-7.789633123003725,-4.583203975644266) -- (-7.281423690604183,-0.9063605673234938) -- (-5.441177776404202,0.1500026822859417) -- (-3.600931862204221,1.2063659318953772) -- (-3.0927224298046796,4.883209340216149) -- (-1.2524765156046986,5.939572589825585) -- (0.5877693985952823,6.99593583943502) -- (1.0959788309948237,10.672779247755791) -- (2.9362247451948047,11.729142497365228) -- (3.444434177594346,15.405985905685998) -- (5.284680091794327,16.462349155295435) -- cycle;
\draw [line width=1.pt,color=qqqqff] (-12.05375574400662,-3.3265569402637114)-- (-11.54554631160708,0.35028646805706076);
\draw [line width=1.pt,color=qqqqff] (-9.121669604804165,-1.9627238169329293)-- (-9.629879037203706,-5.6395672252537015);
\draw [line width=1.pt,color=qqqqff] (-7.281423690604183,-0.9063605673234938)-- (-7.789633123003725,-4.583203975644266);
\draw [line width=1.pt,color=qqqqff] (-5.441177776404202,0.1500026822859417)-- (-4.93296834400466,3.8268460906067134);
\draw [line width=1.pt,color=qqqqff] (-7.865054483207117,2.4630129672759318)-- (-7.356845050807575,6.139856375596704);
\draw [line width=1.pt,color=qqqqff] (-9.705300397407097,1.4066497176664963)-- (-9.197090965007556,5.083493125987268);
\draw [line width=1.pt,color=qqqqff] (-12.129177104210012,3.7196600026564863)-- (-11.62096767181047,7.3965034109772585);
\draw [line width=1.pt,color=qqqqff] (-11.62096767181047,7.3965034109772585)-- (-11.11275823941093,11.073346819298031);
\draw [line width=1.pt,color=qqqqff] (-11.696389032013862,14.442720353897453)-- (-11.18817959961432,18.119563762218224);
\draw [line width=1.pt,color=qqqqff] (-9.347933685414342,19.17592701182766)-- (-9.856143117813883,15.49908360350689);
\draw [line width=1.pt,color=qqqqff] (-9.272512325210947,12.129710068907464)-- (-9.78072175761049,8.452866660586693);
\draw [line width=1.pt,color=qqqqff] (-6.924056978611427,16.86291672683767)-- (-7.432266411010968,13.1860733185169);
\draw [line width=1.pt,color=qqqqff] (-5.083811064411446,17.919279976447108)-- (-5.592020496810987,14.242436568126337);
\draw [line width=1.pt,color=qqqqff] (-5.008389704208053,10.873063033526911)-- (-5.516599136607594,7.1962196252061394);
\draw [line width=1.pt,color=qqqqff] (-6.848635618408033,9.816699783917475)-- (-7.356845050807575,6.139856375596704);
\draw [line width=1.pt,color=qqqqff] (-4.857546983801268,-3.219370852313484)-- (-5.36575641620081,-6.896214260634256);
\draw [line width=1.pt,color=qqqqff] (-3.017301069601287,-2.1630076027040483)-- (-3.5255105020008286,-5.8398510110248205);
\draw [line width=1.pt,color=qqqqff] (-1.1770551554013062,-1.1066443530946128)-- (-0.6688457230017648,2.570199055226159);
\draw [line width=1.pt,color=qqqqff] (-3.168143790008072,11.929426283136348)-- (-2.6599343576085306,15.606269691457118);
\draw [line width=1.pt,color=qqqqff] (-3.0927224298046796,4.883209340216149)-- (-2.584512997405138,8.560052748536922);
\draw [line width=1.pt,color=qqqqff] (-1.2524765156046986,5.939572589825585)-- (-0.7442670832051572,9.616415998146357);
\draw [line width=1.pt,color=qqqqff] (-1.327897875808091,12.98578953274578)-- (-0.8196884434085496,16.66263294106655);
\draw [line width=1.pt,color=qqqqff] (1.0959788309948237,10.672779247755791)-- (1.6041882633943652,14.349622656076562);
\draw [line width=1.pt,color=qqqqff] (2.9362247451948047,11.729142497365228)-- (3.444434177594346,15.405985905685998);
\draw [line width=1.pt,color=qqqqff] (0.6631907587986747,-0.050281103485177336)-- (1.1714001911982161,3.6265623048355944);
\draw [line width=1.pt,color=qqqqff] (-0.5934243627983724,-4.476017887694038)-- (-1.1016337951979138,-8.15286129601481);
\draw [line width=1.pt,color=qqqqff] (1.2468215514016086,-3.4196546380846025)-- (0.7386121190020671,-7.096498046405374);
\draw [line width=1.pt,color=qqqqff] (3.0870674656015895,-2.363291388475167)-- (3.595276898001131,1.3135520198456048);
\draw [line width=1.pt,color=qqqqff] (5.435522812201112,2.3699152694550403)-- (4.9273133798015705,-1.3069281388657314);
\draw [line width=1.pt,color=qqqqff] (3.6706982582045233,-5.7326649230745925)-- (3.162488825804982,-9.409508331395365);
\draw [line width=1.pt,color=qqqqff] (5.002734740004962,-8.353145081785929)-- (5.510944172404504,-4.676301673465156);
\draw [line width=1.pt,color=qqqqff] (6.842980654204943,-7.296781832176492)-- (7.351190086604485,-3.6199384238557206);
\draw [line width=1.pt,color=qqqqff] (7.351190086604485,-3.6199384238557206)-- (7.859399519004027,0.0569049844650511);
\draw [line width=1.pt,color=qqqqff] (7.275768726401093,3.426278519064476)-- (7.783978158800634,7.1031219273852475);
\draw [line width=1.pt,color=qqqqff] (4.8518920195981785,5.739288804054466)-- (5.3601014519977195,9.416132212375238);
\draw [line width=1.pt,color=qqqqff] (5.3601014519977195,9.416132212375238)-- (5.8683108843972605,13.09297562069601);
\draw [line width=1.pt,color=qqqqff] (3.5198555377977385,8.359768962765802)-- (3.011646105398197,4.6829255544450294);
\draw [line width=1.pt,color=qqqqff] (0.5877693985952823,6.99593583943502)-- (1.0959788309948237,10.672779247755791);
\draw [line width=1.pt,color=qqqqff] (7.2003473661977,10.472495461984675)-- (7.708556798597241,14.149338870305446);
\draw [line width=1.pt,color=qqqqff] (9.624224073000615,8.159485176994682)-- (10.132433505400156,11.836328585315453);
\draw [line width=1.pt,color=qqqqff] (11.464469987200596,9.215848426604119)-- (11.972679419600137,12.89269183492489);
\draw [line width=1.pt,color=qqqqff] (9.116014640601074,4.482641768673911)-- (9.624224073000615,8.159485176994682);
\draw [line width=1.pt,color=qqqqff] (9.191436000804465,-2.563575174246285)-- (9.699645433204008,1.1132682340744866);
\draw [line width=1.pt,color=qqqqff] (11.539891347403989,2.1696314836839217)-- (12.04810077980353,5.846474892004693);
\draw [line width=1.pt,color=qqqqff] (13.38013726160397,3.225994733293357)-- (13.888346694003511,6.9028381416141285);
\draw [line width=1.pt,color=qqqqff] (11.61531270760738,-4.876585459236274)-- (12.123522140006921,-1.1997420509155026);
\draw [line width=1.pt,color=qqqqff] (11.10710327520784,-8.553428867557045)-- (11.61531270760738,-4.876585459236274);
\draw [line width=1.pt,color=qqqqff] (9.266857361007858,-9.609792117166482)-- (9.775066793407401,-5.932948708845711);
\draw [line width=1.pt,color=qqqqff] (13.455558621807361,-3.820222209626839)-- (13.963768054206902,-0.14337880130606706);
\draw [line width=1.pt,color=ccqqqq] (13.530979982010754,-10.866439152547034)-- (14.039189414410295,-7.1895957442262635);
\draw [line width=1.pt,color=ccqqqq] (-7.789633123003725,-4.583203975644266)-- (-9.629879037203706,-5.6395672252537015);
\draw [line width=1.pt,color=ccqqqq] (-7.281423690604183,-0.9063605673234938)-- (-9.121669604804165,-1.9627238169329293);
\draw [line width=1.pt,color=ccqqqq] (-7.865054483207117,2.4630129672759318)-- (-9.705300397407097,1.4066497176664963);
\draw [line width=1.pt,color=ccqqqq] (-9.705300397407097,1.4066497176664963)-- (-11.54554631160708,0.35028646805706076);
\draw [line width=1.pt,color=ccqqqq] (-12.129177104210012,3.7196600026564863)-- (-13.969423018409994,2.663296753047051);
\draw [line width=1.pt,color=ccqqqq] (-9.197090965007556,5.083493125987268)-- (-7.356845050807575,6.139856375596704);
\draw [line width=1.pt,color=ccqqqq] (-7.356845050807575,6.139856375596704)-- (-5.516599136607594,7.1962196252061394);
\draw [line width=1.pt,color=ccqqqq] (-6.848635618408033,9.816699783917475)-- (-5.008389704208053,10.873063033526911);
\draw [line width=1.pt,color=ccqqqq] (-5.008389704208053,10.873063033526911)-- (-3.168143790008072,11.929426283136348);
\draw [line width=1.pt,color=ccqqqq] (-3.168143790008072,11.929426283136348)-- (-1.327897875808091,12.98578953274578);
\draw [line width=1.pt,color=ccqqqq] (-2.584512997405138,8.560052748536922)-- (-0.7442670832051572,9.616415998146357);
\draw [line width=1.pt,color=ccqqqq] (-0.7442670832051572,9.616415998146357)-- (1.0959788309948237,10.672779247755791);
\draw [line width=1.pt,color=ccqqqq] (1.0959788309948237,10.672779247755791)-- (2.9362247451948047,11.729142497365228);
\draw [line width=1.pt,color=ccqqqq] (1.6041882633943652,14.349622656076562)-- (3.444434177594346,15.405985905685998);
\draw [line width=1.pt,color=ccqqqq] (3.444434177594346,15.405985905685998)-- (5.284680091794327,16.462349155295435);
\draw [line width=1.pt,color=ccqqqq] (5.284680091794327,16.462349155295435)-- (7.124926005994308,17.51871240490487);
\draw [line width=1.pt,color=ccqqqq] (5.8683108843972605,13.09297562069601)-- (7.708556798597241,14.149338870305446);
\draw [line width=1.pt,color=ccqqqq] (-9.347933685414342,19.17592701182766)-- (-11.18817959961432,18.119563762218224);
\draw [line width=1.pt,color=ccqqqq] (-13.536634946213844,13.38635710428802)-- (-11.696389032013862,14.442720353897453);
\draw [line width=1.pt,color=ccqqqq] (-11.696389032013862,14.442720353897453)-- (-9.856143117813883,15.49908360350689);
\draw [line width=1.pt,color=ccqqqq] (-11.11275823941093,11.073346819298031)-- (-9.272512325210947,12.129710068907464);
\draw [line width=1.pt,color=ccqqqq] (-7.432266411010968,13.1860733185169)-- (-5.592020496810987,14.242436568126337);
\draw [line width=1.pt,color=ccqqqq] (-7.432266411010968,13.1860733185169)-- (-9.272512325210947,12.129710068907464);
\draw [line width=1.pt,color=ccqqqq] (-6.924056978611427,16.86291672683767)-- (-5.083811064411446,17.919279976447108);
\draw [line width=1.pt,color=ccqqqq] (-2.6599343576085306,15.606269691457118)-- (-0.8196884434085496,16.66263294106655);
\draw [line width=1.pt,color=ccqqqq] (-11.62096767181047,7.3965034109772585)-- (-9.78072175761049,8.452866660586693);
\draw [line width=1.pt,color=ccqqqq] (-7.281423690604183,-0.9063605673234938)-- (-5.441177776404202,0.1500026822859417);
\draw [line width=1.pt,color=ccqqqq] (-5.441177776404202,0.1500026822859417)-- (-3.600931862204221,1.2063659318953772);
\draw [line width=1.pt,color=ccqqqq] (-3.0927224298046796,4.883209340216149)-- (-4.93296834400466,3.8268460906067134);
\draw [line width=1.pt,color=ccqqqq] (-4.857546983801268,-3.219370852313484)-- (-3.017301069601287,-2.1630076027040483);
\draw [line width=1.pt,color=ccqqqq] (-1.2524765156046986,5.939572589825585)-- (0.5877693985952823,6.99593583943502);
\draw [line width=1.pt,color=ccqqqq] (-3.0927224298046796,4.883209340216149)-- (-1.2524765156046986,5.939572589825585);
\draw [line width=1.pt,color=ccqqqq] (3.5198555377977385,8.359768962765802)-- (5.3601014519977195,9.416132212375238);
\draw [line width=1.pt,color=ccqqqq] (5.3601014519977195,9.416132212375238)-- (7.2003473661977,10.472495461984675);
\draw [line width=1.pt,color=ccqqqq] (7.708556798597241,14.149338870305446)-- (9.548802712797222,15.205702119914882);
\draw [line width=1.pt,color=ccqqqq] (10.132433505400156,11.836328585315453)-- (11.972679419600137,12.89269183492489);
\draw [line width=1.pt,color=ccqqqq] (7.783978158800634,7.1031219273852475)-- (9.624224073000615,8.159485176994682);
\draw [line width=1.pt,color=ccqqqq] (9.624224073000615,8.159485176994682)-- (11.464469987200596,9.215848426604119);
\draw [line width=1.pt,color=ccqqqq] (3.011646105398197,4.6829255544450294)-- (4.8518920195981785,5.739288804054466);
\draw [line width=1.pt,color=ccqqqq] (-0.6688457230017648,2.570199055226159)-- (1.1714001911982161,3.6265623048355944);
\draw [line width=1.pt,color=ccqqqq] (1.1714001911982161,3.6265623048355944)-- (3.011646105398197,4.6829255544450294);
\draw [line width=1.pt,color=ccqqqq] (-1.1770551554013062,-1.1066443530946128)-- (0.6631907587986747,-0.050281103485177336);
\draw [line width=1.pt,color=ccqqqq] (-3.017301069601287,-2.1630076027040483)-- (-1.1770551554013062,-1.1066443530946128);
\draw [line width=1.pt,color=ccqqqq] (-5.36575641620081,-6.896214260634256)-- (-3.5255105020008286,-5.8398510110248205);
\draw [line width=1.pt,color=ccqqqq] (-0.5934243627983724,-4.476017887694038)-- (1.2468215514016086,-3.4196546380846025);
\draw [line width=1.pt,color=ccqqqq] (1.2468215514016086,-3.4196546380846025)-- (3.0870674656015895,-2.363291388475167);
\draw [line width=1.pt,color=ccqqqq] (3.595276898001131,1.3135520198456048)-- (5.435522812201112,2.3699152694550403);
\draw [line width=1.pt,color=ccqqqq] (5.435522812201112,2.3699152694550403)-- (7.275768726401093,3.426278519064476);
\draw [line width=1.pt,color=ccqqqq] (7.275768726401093,3.426278519064476)-- (9.116014640601074,4.482641768673911);
\draw [line width=1.pt,color=ccqqqq] (3.0870674656015895,-2.363291388475167)-- (4.9273133798015705,-1.3069281388657314);
\draw [line width=1.pt,color=ccqqqq] (7.859399519004027,0.0569049844650511)-- (9.699645433204008,1.1132682340744866);
\draw [line width=1.pt,color=ccqqqq] (9.699645433204008,1.1132682340744866)-- (11.539891347403989,2.1696314836839217);
\draw [line width=1.pt,color=ccqqqq] (12.04810077980353,5.846474892004693)-- (13.888346694003511,6.9028381416141285);
\draw [line width=1.pt,color=ccqqqq] (11.539891347403989,2.1696314836839217)-- (13.38013726160397,3.225994733293357);
\draw [line width=1.pt,color=ccqqqq] (3.6706982582045233,-5.7326649230745925)-- (5.510944172404504,-4.676301673465156);
\draw [line width=1.pt,color=ccqqqq] (5.510944172404504,-4.676301673465156)-- (7.351190086604485,-3.6199384238557206);
\draw [line width=1.pt,color=ccqqqq] (7.351190086604485,-3.6199384238557206)-- (9.191436000804465,-2.563575174246285);
\draw [line width=1.pt,color=ccqqqq] (12.123522140006921,-1.1997420509155026)-- (13.963768054206902,-0.14337880130606706);
\draw [line width=1.pt,color=ccqqqq] (13.963768054206902,-0.14337880130606706)-- (15.804013968406883,0.9129844483033684);
\draw [line width=1.pt,color=ccqqqq] (-1.1016337951979138,-8.15286129601481)-- (0.7386121190020671,-7.096498046405374);
\draw [line width=1.pt,color=ccqqqq] (3.162488825804982,-9.409508331395365)-- (5.002734740004962,-8.353145081785929);
\draw [line width=1.pt,color=ccqqqq] (5.002734740004962,-8.353145081785929)-- (6.842980654204943,-7.296781832176492);
\draw [line width=1.pt,color=ccqqqq] (9.775066793407401,-5.932948708845711)-- (11.61531270760738,-4.876585459236274);
\draw [line width=1.pt,color=ccqqqq] (11.61531270760738,-4.876585459236274)-- (13.455558621807361,-3.820222209626839);
\draw [line width=1.pt,color=ccqqqq] (9.266857361007858,-9.609792117166482)-- (11.10710327520784,-8.553428867557045);
\draw [line width=1.pt,color=zzttff] (14.039189414410295,-7.1895957442262635)-- (15.879435328610276,-6.133232494616829);
\draw [line width=1.pt,color=zzttff] (-13.536634946213844,13.38635710428802)-- (-11.11275823941093,11.073346819298031);
\draw [line width=1.pt,color=zzttff] (-11.696389032013862,14.442720353897453)-- (-9.272512325210947,12.129710068907464);
\draw [line width=1.pt,color=zzttff] (-9.856143117813883,15.49908360350689)-- (-7.432266411010968,13.1860733185169);
\draw [line width=1.pt,color=zzttff] (-9.347933685414342,19.17592701182766)-- (-6.924056978611427,16.86291672683767);
\draw [line width=1.pt,color=zzttff] (-9.272512325210947,12.129710068907464)-- (-6.848635618408033,9.816699783917475);
\draw [line width=1.pt,color=zzttff] (-13.969423018409994,2.663296753047051)-- (-11.54554631160708,0.35028646805706076);
\draw [line width=1.pt,color=zzttff] (-12.129177104210012,3.7196600026564863)-- (-9.705300397407097,1.4066497176664963);
\draw [line width=1.pt,color=zzttff] (-11.62096767181047,7.3965034109772585)-- (-9.197090965007556,5.083493125987268);
\draw [line width=1.pt,color=zzttff] (-9.78072175761049,8.452866660586693)-- (-7.356845050807575,6.139856375596704);
\draw [line width=1.pt,color=zzttff] (-7.432266411010968,13.1860733185169)-- (-5.008389704208053,10.873063033526911);
\draw [line width=1.pt,color=zzttff] (-5.592020496810987,14.242436568126337)-- (-3.168143790008072,11.929426283136348);
\draw [line width=1.pt,color=zzttff] (-5.083811064411446,17.919279976447108)-- (-2.6599343576085306,15.606269691457118);
\draw [line width=1.pt,color=zzttff] (-0.8196884434085496,16.66263294106655)-- (1.6041882633943652,14.349622656076562);
\draw [line width=1.pt,color=zzttff] (3.444434177594346,15.405985905685998)-- (5.8683108843972605,13.09297562069601);
\draw [line width=1.pt,color=zzttff] (5.284680091794327,16.462349155295435)-- (7.708556798597241,14.149338870305446);
\draw [line width=1.pt,color=zzttff] (7.124926005994308,17.51871240490487)-- (9.548802712797222,15.205702119914882);
\draw [line width=1.pt,color=zzttff] (9.548802712797222,15.205702119914882)-- (11.972679419600137,12.89269183492489);
\draw [line width=1.pt,color=zzttff] (7.708556798597241,14.149338870305446)-- (10.132433505400156,11.836328585315453);
\draw [line width=1.pt,color=zzttff] (7.2003473661977,10.472495461984675)-- (9.624224073000615,8.159485176994682);
\draw [line width=1.pt,color=zzttff] (11.464469987200596,9.215848426604119)-- (13.888346694003511,6.9028381416141285);
\draw [line width=1.pt,color=zzttff] (9.624224073000615,8.159485176994682)-- (12.04810077980353,5.846474892004693);
\draw [line width=1.pt,color=zzttff] (-1.327897875808091,12.98578953274578)-- (1.0959788309948237,10.672779247755791);
\draw [line width=1.pt,color=zzttff] (2.9362247451948047,11.729142497365228)-- (5.3601014519977195,9.416132212375238);
\draw [line width=1.pt,color=zzttff] (1.0959788309948237,10.672779247755791)-- (3.5198555377977385,8.359768962765802);
\draw [line width=1.pt,color=zzttff] (-3.168143790008072,11.929426283136348)-- (-0.7442670832051572,9.616415998146357);
\draw [line width=1.pt,color=zzttff] (-5.008389704208053,10.873063033526911)-- (-2.584512997405138,8.560052748536922);
\draw [line width=1.pt,color=zzttff] (-5.516599136607594,7.1962196252061394)-- (-3.0927224298046796,4.883209340216149);
\draw [line width=1.pt,color=zzttff] (-7.356845050807575,6.139856375596704)-- (-4.93296834400466,3.8268460906067134);
\draw [line width=1.pt,color=zzttff] (-7.865054483207117,2.4630129672759318)-- (-5.441177776404202,0.1500026822859417);
\draw [line width=1.pt,color=zzttff] (-9.705300397407097,1.4066497176664963)-- (-7.281423690604183,-0.9063605673234938);
\draw [line width=1.pt,color=zzttff] (-11.54554631160708,0.35028646805706076)-- (-9.121669604804165,-1.9627238169329293);
\draw [line width=1.pt,color=zzttff] (-12.05375574400662,-3.3265569402637114)-- (-9.629879037203706,-5.6395672252537015);
\draw [line width=1.pt,color=zzttff] (-7.789633123003725,-4.583203975644266)-- (-5.36575641620081,-6.896214260634256);
\draw [line width=1.pt,color=zzttff] (-7.281423690604183,-0.9063605673234938)-- (-4.857546983801268,-3.219370852313484);
\draw [line width=1.pt,color=zzttff] (-5.441177776404202,0.1500026822859417)-- (-3.017301069601287,-2.1630076027040483);
\draw [line width=1.pt,color=zzttff] (-3.600931862204221,1.2063659318953772)-- (-1.1770551554013062,-1.1066443530946128);
\draw [line width=1.pt,color=zzttff] (-3.0927224298046796,4.883209340216149)-- (-0.6688457230017648,2.570199055226159);
\draw [line width=1.pt,color=zzttff] (-1.2524765156046986,5.939572589825585)-- (1.1714001911982161,3.6265623048355944);
\draw [line width=1.pt,color=zzttff] (0.5877693985952823,6.99593583943502)-- (3.011646105398197,4.6829255544450294);
\draw [line width=1.pt,color=zzttff] (-3.0927224298046796,4.883209340216149)-- (-3.600931862204221,1.2063659318953772);
\draw [line width=1.pt,color=zzttff] (-3.5255105020008286,-5.8398510110248205)-- (-1.1016337951979138,-8.15286129601481);
\draw [line width=1.pt,color=zzttff] (-3.017301069601287,-2.1630076027040483)-- (-0.5934243627983724,-4.476017887694038);
\draw [line width=1.pt,color=zzttff] (-1.1770551554013062,-1.1066443530946128)-- (1.2468215514016086,-3.4196546380846025);
\draw [line width=1.pt,color=zzttff] (0.6631907587986747,-0.050281103485177336)-- (3.0870674656015895,-2.363291388475167);
\draw [line width=1.pt,color=zzttff] (1.1714001911982161,3.6265623048355944)-- (3.595276898001131,1.3135520198456048);
\draw [line width=1.pt,color=zzttff] (3.011646105398197,4.6829255544450294)-- (5.435522812201112,2.3699152694550403);
\draw [line width=1.pt,color=zzttff] (4.8518920195981785,5.739288804054466)-- (7.275768726401093,3.426278519064476);
\draw [line width=1.pt,color=zzttff] (5.3601014519977195,9.416132212375238)-- (7.783978158800634,7.1031219273852475);
\draw [line width=1.pt,color=zzttff] (0.7386121190020671,-7.096498046405374)-- (3.162488825804982,-9.409508331395365);
\draw [line width=1.pt,color=zzttff] (1.2468215514016086,-3.4196546380846025)-- (3.6706982582045233,-5.7326649230745925);
\draw [line width=1.pt,color=zzttff] (3.0870674656015895,-2.363291388475167)-- (5.510944172404504,-4.676301673465156);
\draw [line width=1.pt,color=zzttff] (4.9273133798015705,-1.3069281388657314)-- (7.351190086604485,-3.6199384238557206);
\draw [line width=1.pt,color=zzttff] (5.435522812201112,2.3699152694550403)-- (7.859399519004027,0.0569049844650511);
\draw [line width=1.pt,color=zzttff] (7.275768726401093,3.426278519064476)-- (9.699645433204008,1.1132682340744866);
\draw [line width=1.pt,color=zzttff] (9.116014640601074,4.482641768673911)-- (11.539891347403989,2.1696314836839217);
\draw [line width=1.pt,color=zzttff] (13.38013726160397,3.225994733293357)-- (15.804013968406883,0.9129844483033684);
\draw [line width=1.pt,color=zzttff] (11.539891347403989,2.1696314836839217)-- (13.963768054206902,-0.14337880130606706);
\draw [line width=1.pt,color=zzttff] (9.699645433204008,1.1132682340744866)-- (12.123522140006921,-1.1997420509155026);
\draw [line width=1.pt,color=zzttff] (9.191436000804465,-2.563575174246285)-- (11.61531270760738,-4.876585459236274);
\draw [line width=1.pt,color=zzttff] (13.455558621807361,-3.820222209626839)-- (15.879435328610276,-6.133232494616829);
\draw [line width=1.pt,color=zzttff] (11.61531270760738,-4.876585459236274)-- (14.039189414410295,-7.1895957442262635);
\draw [line width=1.pt,color=zzttff] (11.10710327520784,-8.553428867557045)-- (13.530979982010754,-10.866439152547034);
\draw [line width=1.pt,color=zzttff] (9.775066793407401,-5.932948708845711)-- (7.351190086604485,-3.6199384238557206);
\draw [line width=1.pt,color=zzttff] (6.842980654204943,-7.296781832176492)-- (9.266857361007858,-9.609792117166482);
\end{tikzpicture}
  \caption{Tiling corresponding to the multigrid of Figure \ref{fig:test1}. Coloured “rails” of tiles correspond to the highlighted lines from Figure \ref{fig:test1}.}
  \label{fig:test2}
\end{minipage}
\end{figure}

\subsubsection{Quasicrystal tiling in the “primal space”}\label{sec_qctiling}

 We next associate to a multigrid as in \eqref{mgrid} a tiling of $\mathbb{R}^d$ by parallelotopes, which will constitute our quasicrystal in the
“primal” space.  In particular to every intersection point in the multigrid 
\begin{equation}\label{eq:xJk}
x=x(J,k_J)=\bigcap_{g\in J} H(g,k_g,\gamma_g)\qquad J\subset \G,\, \sharp J=d,\quad k\in \Z^J
\end{equation}
we associate a parallelotope $P(x)$ of the tiling. 
\medskip

{\bf Convention. } We will interpret $J\subset \G$ not just as a set of vectors $g$ but as an \emph{ordered set}. This allows to denote without ambiguity $\det(J):=\det(g_1,\dots,g_d)$, if $J=\{g_1,\dots,g_d\}$. To do this, we fix once and for all an ordering of $\G$ and for each subset $J\subset\G$ we will always order its elements in increasing order with respect to this ordering.
\medskip

Before explaining the construction we need to introduce a further datum, namely a map 
$\mathcal G\ni g\mapsto \widetilde g\in \R^d$ (the collection of $\widetilde g$ will give the possible directions of the $1$-dimensional edges of the parallelotopes $P(x)$ cf. Figure \ref{fig:test2}), on which we require the following compatibility condition: 
\begin{equation}\label{detcond}
\det(g_1,\ldots,g_d)\det(\widetilde g_1,\ldots,\widetilde g_d)>0 \text{ for any choice }g_1,\dots g_d\in \mathcal G.
\end{equation}
Note that this implies that $g\mapsto \widetilde g$ is
bijective.
For the constrution it would be equivalent to require the product to be always negative, the important thing being that the sign does not depend on $g_1,\ldots,g_d$.
We observe that a standard choice to produce Penrose tilings in the plane is to set $\widetilde g=g^\perp$.

\medskip
The vectors $\widetilde g$, in conjunction with the multigrid $M(\mathcal G, \gamma)$ above,
define a tiling of $\R^d$ as follows. To every $x=x(J,k_J)$ we associate the unique vector $k=k_{\mathcal G}\in\Z^{\mathcal G}$ which “extends” $k_J$ and satisfies the {\it admissibility condition}
\begin{equation}\label{ptinslab}
    x(J,k_J)\in\bigcap_{g\in \mathcal G\setminus J} S_{g,k_g},
\end{equation}
and we consider the parallelotope $P(x)\subset \R^d$ defined as a Minkowski sum as follows:
\begin{equation}\label{parallelotopes}
    P(x):=\sum_{g\in\mathcal G}k_g\widetilde g + \sum_{g\in J}[0,1]\widetilde g\qquad\text{where $x=x(J,k_J)$ and $k\in\Z^\G$ satisfies \eqref{ptinslab}} 
\end{equation}
The condition \eqref{detcond} ensures that the union of all such tiles is a tiling of $\R^d$. This is claimed in \cite[p. 270]{GahRhy86} for dimensions up to $3$, where the authors say that they believe the same criterion holds in general dimension but did not check this in detail. In the case of Penrose tilings, a proof is also present in \cite{de1981algebraic}, for very specific choices $g, \widetilde g$. We provide a proof in the general case in Proposition \ref{prop:qctiling}.
\medskip

{\bf Notations.} 
We denote the set of all possible vertices $x(J, k_J)$ as in
\eqref{eq:xJk} by $\mathcal X=\mathcal X(\mathcal G,\gamma)$ and the set of tiles 
$P(x)$ defined as in \eqref{parallelotopes} associated to 
$x\in\mathcal X(\mathcal G,\gamma)$ by 
$\mathcal T=\mathcal T(\mathcal G,\gamma,\widetilde{\mathcal G})$.

\subsubsection{Bounded distortion between primal and dual space configurations}\label{sec_bddist}
For $X,Y\subset \R^d$ a map $\phi:X\to Y$ has {\it bounded distortion} (or is a $BD$ map) 
if $\sup_{x\in X}|\phi(x)-x|<+\infty$. We say that $Y$ is $BD$ to $X$ if 
there exists a $BD$ bijection $\phi:X\to Y$. Following the above notation, 
we will show that the bijection that relates the points $x(J, k_J)$ 
as in \eqref{ptinslab} to the centers of the corresponding parallelotopes
\eqref{parallelotopes} from the associated tiling, is a $BD$ map, up to an affine distortion:

\begin{lemma}\label{bda}
Let $\mathcal X$ and $\mathcal T$ be the admissible multigrid points and tiles 
associated to $M(\mathcal G, \gamma)$ and to $\widetilde{\mathcal G}$ as above.
Let $A:\R^d\to\R^d $ be the affine map defined by
\begin{equation}\label{eq:A}
Ax=\sum_{g\in\G} \tfrac{1}{|g|}(\langle x,\tfrac{g}{|g|}\rangle -\gamma_g)\widetilde g
\end{equation}
and let $\phi:\mathcal X\to\R^d$ be the map associating to a vertex $x\in \mathcal X$ the center of the corresponding parallelotope $P(x)$, namely
\[
\phi(x)=\sum_{g\in\mathcal G }k_g\widetilde g + \frac12\sum_{g\in J}\widetilde g
\]
where $x=x(J,k_J)$ as in \eqref{eq:xJk} and $k\in\Z^\G$ satisfies \eqref{ptinslab}.
Then 
\[
\sup_{x\in\mathcal X}|\phi(x)-Ax|<\infty.
\]
Moreover, assuming the condition \eqref{detcond}, the map $A$ is invertible.
\end{lemma}
\begin{proof}
From the compatibility condition \eqref{detcond} and the Cauchy-Binet formula for the determinant we immediately obtain the invertibility of $A$. Indeed we can write
\[
Ax=\sum_{g\in\mathcal G}\left(\left\langle x, \frac{g}{|g|}\right\rangle - \gamma_g\right)\frac{\widetilde g}{|g|}=\mathbb{G}\mathbb{\widetilde{ G}}^\top x -\sum_{g\in\mathcal G}\frac{\gamma_g}{|g|}\widetilde g,
\]
where $\mathbb G$ and $\mathbb{ \widetilde {G}}$ are the $d\times (\sharp \G)$ matrices with columns respectively running through $(\frac{g}{|g|})_{g\in\G}$ and $(\frac{\widetilde g}{|g|})_{g\in \G}$ (in the same order with respect to the mapping $\sim$). Now Cauchy-Binet formula gives that
\begin{equation}\label{eq:CauchyBinet}
\det(\mathbb G \mathbb{\widetilde{G}}^\top)=\sum_{J} \det(\mathbb G_J)\det(\mathbb{\widetilde{G}}_J)
\end{equation}
where $J$ runs through all subsets of $\{1,\ldots,\sharp \G\}$ with $d$ elements and where $\mathbb G_J$ stands for the $d\times d$ minor of $\mathbb G$ obtained considering only the columns corresponding to $J$. The condition \eqref{detcond} implies that every term in the right hand side of the above expression has the same sign (the sign of the determinant does not change if we multiply every column by a positive factor), and therefore we obtain the invertibility of $A$.

Next, if $x$ satisfies \eqref{ptinslab} then $k_g=\lfloor \tfrac{1}{|g|}\langle x,\tfrac{g}{|g|}\rangle-\gamma_g \rfloor$. As a consequence, denoting by $\{\cdot\}$ the fractional part, we have
\begin{eqnarray}\label{tildeayreal}
\phi(x)-Ax&=&\sum_{g\in\mathcal G\setminus J} \left\{\tfrac{1}{|g|}\langle x,\tfrac{g}{|g|}\rangle-\gamma_g \right\}\widetilde g+ \frac12\sum_{g\in J}\widetilde g\nonumber\\
 &\in& \mathrm{Int}\left(\sum_{g\in\mathcal G\setminus J}[0,1]\widetilde g 
 + \frac12\sum_{g\in J}\widetilde g\right).
 \end{eqnarray}
 The maximum modulus of the right hand side of \eqref{tildeayreal} 
is achieved at one of the vertices of the convex polytope on the right, 
 and this leads directly to
\[
\sup_{x\in \mathcal X} |\phi(x)-Ax|\leq  \max\left\{ \left|\sum_{g'\in J'}\widetilde g'
            +\frac12\sum_{g\in J}\widetilde g\right|:\ 
            \begin{array}{l}J,J'\subset\mathcal G\\ 
            J\cap J'=\emptyset\\ \sharp J=d\end{array}
    \right\}.
\]
The conclusion follows.
\end{proof}

The previous result allows us to transfer the problem from the dual to the primal space with a finite error. Roughly speaking, since we rescale down the space to find the Wulff shape, this finite error asymptotically vanishes in the rescaling, and this implies that the Wulff shape in the dual and in the primal space will be the same, up to an affine map.

\medskip
We next prove the following:

\begin{proposition}\label{prop:qctiling}
If \eqref{detcond} holds then the parallelotopes \eqref{parallelotopes} form a tiling of $\mathbb R^d$ which generates a polyhedral complex dual to the one generated by the hyperplane arrangement $M(\mathcal C, \gamma)$.
\end{proposition}
The beginning of the proof is standard within the theory of polyhedral complexes, however we include it for the sake of keeping the presentation self-contained and more transparent.
\begin{proof} Let $K$ be the polyhedral complex generated by the hyperplane arrangement $\{H_{k,g}: k\in \Z^{\mathcal G}\}$. We follow a standard construction in order to fix an explicit realization of the dual complex of $K$ in the dual space. To do this, first consider the barycentric subdivision $K_1$ of $K$, i.e. the rectilinear simplicial complex with one vertex at the barycenter of each $k$-cell of $K$, for each $k=0,\dots,d$, and whose $d$-cells are formed by the barycenters of $f_0,\dots,f_d$ for all choices of facets $f_0,\dots,f_d$ of $K$ such that $f_0\subsetneq f_1\subsetneq\cdots\subsetneq f_d$. Now the $d$-cell of the dual complex $K^*$ of $K$ corresponding to vertex $x$ can be identified with the union of cells of $K_1$ which contain $x$; these cells of $K_1$ form the so-called “star” of $x$, denoted $\mathrm{Star}(x)$. Lower-dimensional cells of $K^*$ can then be obtained by intersecting $d$-cells. 

\medskip Due to the assumption on our multigrid constants $\gamma$, which is chosen so that no more than $d$ hyperplanes $H_{g,k}$ intersect (see the end of Section \ref{ssec:mgds}), the star in $K_1$ of each vertex $x$ of $K$ is a polyhedral cell complex isomorphic to the one given by the barycentric subdivision of the facet complex of parallelotope $P(x)$, denoted by $K_1(P(x))$. This fact can be proved by induction on $d$, together with the fact that the complex with cells $P(x)$ is combinatorially dual to the one generated by the hyperplanes $H_{g,k}$ and we leave the details to the reader. 

\medskip For each $x$ vertex of $K$, from the cell-complex map $\phi_x:\mathrm{Star}(x)\to K_1(P(x))$ we can define a map $\bar \phi_x:\bigcup \mathrm{Star}(x)\to P(x)$ which sends cells to cells, respects cell inclusions and is affine on the interiors of cells of $\mathrm{Star}(x)$ of any dimension. As $\bar\phi_x$ is a piecewise affine bijection, it thus has Jacobian of constant sign. Furthermore, this sign is equal to the sign of $\det(g:\ g\in J)\det(\widetilde g:\ g\in J)$, which is constant by assumption \eqref{detcond}.

\medskip We now glue the $\bar \phi_x$ to find a continuous piecewise affine map $\bar\phi:\mathbb R^d\to \mathbb R^d$. For this map to be continuous, we need to verify a compatibility condition. Let $x\neq x'\in \mathcal X$ be such that $\mathrm{Star}(x)\cap\mathrm{Star}(x')\neq\emptyset$. Equivalently, $x,x'$ are vertices of the same cell of $K$. In this case, let $j$ be the lowest dimension of a cell of $K$ containing both $x$ and $x'$. This means that we find an index $k^{\mathcal G}\in\mathbb Z^{\mathcal G}$ such that the slabs $S_{g,k_g}, g\in\mathcal G$ as in \eqref{mgslab} all contain $x,x'$ either in their interior or in their boundaries, and there are precisely $j$ choices of $g\in\mathcal G$ such that $x,x'$ are in the same boundary component of $S_{g,k_g}$. From definition \eqref{parallelotopes} it follows that $P(x)\cap P(x')$ have a common $(d-j)$-dimensional facet. By further inspection of indices, we see that then both $\phi_x$ and $\phi_{x'}$ send $\mathrm{Star}(x)\cap \mathrm{Star}(x')$ to the barycentric subdivision of this common facet, and thus $\bar\phi_x$ and $\bar\phi_{x'}$ coincide on $\bigcup(\mathrm{Star}(x)\cap \mathrm{Star}(x'))$, as desired.

\medskip By estimating bilipschitz constants, we find that $\bar\phi$ is at bounded distance from $\phi$ from Lemma \ref{bda}, and this shows that $\bar\phi$ is at bounded distance from the affine bijective map $A$ from the same lemma, and thus it is surjective. Thus $\bar\phi$ is a piecewise affine surjective map of $\mathbb R^d$ such that for each $x\in\mathcal X$ the restriction $\bar\phi|_{\bigcup\mathrm{Star}(x)}$ is bilipschitz with its image and (by \eqref{detcond}) preserves orientation. 

\medskip If we had a continuous covering from $\R^d$ to $\R^d$, it would have to be injective because $\R^d$ is simply connected. For our $\bar \phi$, we cannot ensure that it is a homeomorphism locally near the boundary of $\bigcup \mathrm{Star}(x)$ for $x\in\mathcal X$, however the degree-theoretic proof based on preserved orientations still works, as follows. If we take a large cycle $\Sigma=\partial M$ in $K$ around $y\in\mathbb R^d$, we have that the winding degree satisfies $\mathrm{deg}(\bar\phi, y,\partial M)=1$, where $\mathrm{deg}(\bar\phi, y,\partial M)$ is defined as the winding number of $\frac{\bar\phi(y')-\bar\phi(y)}{|\bar\phi(y')-\bar\phi(y)|}$ around the origin as $y'\in \partial M$. The fact that this degree is $1$ for $\partial M$ far a way from $y$ follows by approximation, from the fact that $\bar\phi$ is a bounded distortion of an affine bijective map. Thus by degree theory (for a proof see \cite[Thm. 4]{lipman2014bijective}) we have for any 
\[
\sharp \{\bar\phi^{-1}(y)\}\le \mathrm{deg}(\bar\phi, y,\partial M)=1,
\]
and thus bijectivity of $\bar\phi$ follows, and in particular the $P(x),x\in\mathcal X$ give a tiling, as desired.
\end{proof}

\subsubsection{Multigrid lines and parallelotope rails}\label{ssec_rails}

We will denote the normalized directions of lines obtained by intersecting $d-1$
hyperplanes of $M(\mathcal G,\gamma)$ as follows:
\begin{equation}\label{lines}
    \mathcal V:=\left\{ v\in \R^d, |v|=1\left|\begin{array}{cc} 
    \exists g_1,\ldots,g_{d-1}\in \mathcal G,\ \exists c\in\R^d\\[2mm]
    \bigcap_{j=1}^{d-1}H_{0,g_j}=\R v+c,\mbox{ and }
    \det(g_1,\ldots,g_{d-1},v)>0,\end{array}\right.\right\}.
\end{equation}
Due to condition \eqref{detcond}, in fact for each line 
$\R v+c, v\in\mathcal V$, there exists exactly one choice of 
$J'_v:=\{g_1,\ldots,g_{d-1}\}$ as in \eqref{lines} and these vectors are linearly independent, therefore the normalisation condition on $v$ in \eqref{lines} is well-posed. 
Note that the requirement that there exists $c$ such that $ \bigcap_{j=1}^{d-1}H_{0,g_j}=\R v+c$
is equivalent to $v\perp g_j$ for $j=1,\ldots,d-1$.

\medskip

We note the following further properties:
\begin{enumerate}
    \item If $v\in\mathcal V$ has uniquely associated hyperplanes directions
    $J'_v=\{g_1,\ldots,g_{d-1}\}\subset \mathcal G$ as in \eqref{lines} (with indices respecting the induced ordering fixed on $\mathcal G$) then to $v\in\mathcal V$ 
    we associate in the primal space the vector $\widetilde v\in\R^d$ determined by the conditions
    \[
    \widetilde v\perp\widetilde g_j\quad\text{for } j=1,\ldots,d-1, 
    \quad |\widetilde v|=1\quad \mbox{ and }\quad  
    \det(\widetilde g_1,\ldots,\widetilde g_{d-1},\widetilde v)>0
    \]
    Again this is a good definition because, as seen before, the mapping 
    $\mathcal G\ni g\mapsto \widetilde g$ is bijective and because we assumed 
    the condition \eqref{detcond}. It follows that $\{\widetilde v:\ v\in\mathcal V\}$ are the normal vectors 
    of the faces of parallelotopes $P(k)$ as in \eqref{parallelotopes}, where $\widetilde v$ 
    is normal to the faces generated by $\widetilde g_1,\ldots, \widetilde g_{d-1}$. Also, 
    all normal vectors to the faces of $P(k)$ are generated in this way, up to a sign. 
    
    \item With the above notation, the successive intersections of $\R v$ with the 
    hyperplanes $H_{g,k}, g\in\mathcal G\setminus\{g_j\}_{j=1}^{d-1}$ correspond to 
    “parallelotope rails”, i.e. chains of parallelotopes $T=P(x)\in\mathcal T$ 
    (see \eqref{parallelotopes}), 
    in which neighboring parallelotopes have in common exactly the face with normal $\widetilde v$.
    \item Such rails were described for $d=2,3$ in \cite{IngSte89,Ho+89}. We will also call a {\bf dual rail} the set of collinear points in the dual space corresponding to tiles forming a rail.
\end{enumerate}

Note that dual rail directions are in direct correspondence with the vectors $v\in\mathcal V$, however in the primal space the vector $\widetilde v$ in general is {\bf not} giving the direction of the parallelotope rail corresponding to $v\in\mathcal V$: this direction is instead parallel to $\mathbb G\widetilde{\mathbb G}^T v$.

\subsubsection{Density of multigrid sublattices and of rails}\label{sec:rails}
For each $J\subset\mathcal G$ of cardinality $\sharp J=d$, 
we form a nondegenerate translated lattice $\Lambda_J\subset\R^d$ by intersecting $d$-ples of planes from the families corresponding to $g\in J$ (cf. equation \eqref{1grid}):
\begin{equation}\label{deflambdaj}
    \Lambda_J:=\bigcup_{\vec k\in \Z^J} \bigcap_{g\in J}H_{g,k_g} = p_J + \left(\mathrm{Span}_{\mathbb Z}\left\{\frac g{|g|^2}:\ g\in J\right\}\right)^*,
\end{equation}
in which $\{p_J\}=\cap_{g\in J}H_{g,0}$ is determined by the translation numbers $\gamma_g,g\in J$ from the definition of our multigrid, and the notaton $\Lambda^*:=\{z\in \R^d :\langle z,p\rangle\in\mathbb Z,\ \forall p\in\Lambda\}$ denotes the dual lattice of $\Lambda$. In particular
\begin{equation}\label{eq:det}
|\det \Lambda_J|=\frac{|g_1|\cdots|g_d|}{|\det \mathbb{G}_J|}=\frac{|g_1|^2\cdots|g_d|^2}{|\det(g_1,\ldots,g_d)|}\ ,
\end{equation}
where like in \eqref{eq:CauchyBinet}, $\mathbb{G}_J$ is the matrix whose columns are $(\tfrac{g}{|g|})_{g\in J}$. 

\medskip

We will also use later the following geometric consideration: if $v$ is one of the multigrid directions \eqref{lines} that generate $\Lambda_J$, then $v^\perp$ intersects the multigrid lines parallel to $v$ in a sublattice $\Lambda_{J'_v}$ (with unit cell denoted by  $U_{\Lambda_{J'_v}}$) which is related to $\mathrm{Span}_{\Z}\left\{\frac{g}{|g|^2}:\ g\in J'_v\right\}$ by a formula like \eqref{deflambdaj}. Therefore, analogously to \eqref{eq:det}, for each $J\supset J'_v, \sharp J=d$, there holds:
\begin{equation}\label{uvperp}
\mathcal H^{d-1}(\pi_{v^\perp}U_{\Lambda_J})=\mathcal H^{d-1}(U_{\Lambda_{J'_v}})=\frac{\prod_{g\in J'_v}|g|^2}{|\det(J'_v)|}.
\end{equation}

\medskip

Given a fixed rail direction $v$, for every possible choice of $J$ such that $J\supset J_v'$, the lattice $\Lambda_J$ has a generator $\lambda_{v,J}v$ parallel to $v$, with $\lambda_{v,J}>0$. Furthermore we have
\begin{equation}\label{eq:lambdavJ}
\frac{\det \Lambda_J}{\lambda_{v,J}}=\H^{d-1}(\pi_{v^\perp} U_{\Lambda_J})=\mathcal H^{d-1}(U_{\Lambda_{J'_v}}),
\end{equation}
which shows that the leftmost quantity \textit{does not depend on $J$, but only on $v$}. We will use this fact in the proof of the $\liminf$ inequality.

\medskip

Furthermore, we find that if $P$ is a polygon in the hyperplane $\nu^\perp$, where
$\nu\in\R^d$ is a unit vector, then the number of dual rails that meet $P$ 
per unit area of $P$ (and in the limit of “large $P$”) equals 
\[
    \frac{\langle \nu,v\rangle}{\H^{d-1}(U_{\Lambda_{J'_v}})}.
\]
where $J'_v=\{g_1,\ldots,g_{d-1}\}\subset \mathcal G$ denotes the set
of vectors as in the definition \eqref{lines} of $v\in \mathcal V$. 
The above equality follows from the fact that $v\perp g$ for 
$g\in J'_v$ and $|v|=1$ by the definitions of 
$J', \mathcal V$, and is the multigrid analogue of \eqref{biperp}.

\medskip

The vertices of the multigrid are the union of all $\Lambda_J$ as above,
and since the traslations $\gamma_g$ are such that no more than $d$ planes from 
the multigrid meet simultaneously, then this union is disjoint.

\subsection{Energy functional in primal and dual spaces}\label{ssec:energyqc}
Let $\mathcal T$ be our quasiperiodic tiling of $\mathbb R^d$.
The parallelotopes from $\mathcal T$ have edges which are translations of the segments
$[0, \widetilde g]$ for $g\in\mathcal G$. The $(d-1)$-dimensional facets of elements of $\mathcal T$
have $\mathcal H^{d-1}$-measures equal to $|\det(\widetilde g:\ g\in J')|$\footnote{Here and throughout we somewhat improperly denote by $|\det(v_1,\ldots,v_k)|$ the Euclidean norm of the multivector $v_1\wedge\ldots\wedge v_k$, i.e. the $k$-dimensional measure of the parallelotope with sides $v_1,\ldots,v_k$.}, 
where $J'\subset \mathcal G$, $\sharp J'=d-1$, indexes their set of edges.
\medskip

The same information as determined by the mention of $J'$ 
can be equivalently encoded in terms of the associated rail directions $v_{J'}$, 
i.e. using the observations and notation of Section \ref{ssec_rails}. 
In this case we define a potential on the multigrid directions, $W:\mathcal V\to \R_+$ 
which associates to $v\in \mathcal V$ defined in \eqref{lines}, 
the weighted surface area given by 
\begin{equation}\label{defweight}
W(v):=w(\widetilde v)\ \det(\widetilde g:\ g\in J'_v),
\end{equation}
where $J'_v\subset \mathcal G$ is related to $v\in\mathcal V$ via \eqref{lines}, 
and $w(\widetilde v)$ is the weight which we associate to faces with normal direction $\widetilde v$, which is the “primal space” vector corresponding to $v$. \medskip

Note that here we use the opposite sign convention for $W$ compared to the case of lattices (in which $V$ was nonpositive), which seems more natural in this case.

\medskip

With notation \eqref{defweight} we then define the energy of a set $T\subset\R^d$,
which is a union of finitely many tiles, as follows:
\begin{eqnarray}\label{entile}
    \mathcal E_{W}(T)&:=&\sum_{v\in\mathcal V} W(v) \ 
    \sharp\{\mbox{facets of $\partial T$ with exterior normal $\widetilde v$}\}\nonumber\\
    &=&\int_{\partial T} w(\nu(x))d\mathcal H^{d-1}(x).
\end{eqnarray}

\subsubsection{Energy in the dual space}
To a finite union $T$ of tiles from $\mathcal T$ there corresponds a finite set of vertices
$X\subset\mathcal X$ and to faces of tiles of $\mathcal T$ there correspond edges between
vertices in $\mathcal X$. We can rewrite the energy \eqref{entile} in terms of the dual 
space elements as 
\begin{equation}\label{entile2}
    \mathcal E_{W}(X):=\sum_{v\in\mathcal V}W(v)
    \sharp\left\{\{x,y\}\subset \mathcal X\left|\begin{array}{l} x-y=\lambda v, \lambda\neq 0,\\
    (x,y)\cap\mathcal X=\emptyset,\\
    \sharp(\{x,y\}\cap X)=1\end{array}\right.\right\}=\mathcal E_{W}(T),
\end{equation}
where $(x,y)$ is the open segment in $\mathbb R^d$ with endpoints $x,y$. To prove the last equality in \eqref{entile2}, note that for each $\widetilde v$, in \eqref{entile} we are counting pairs $\{x,y\}$ corresponding to adjacent tiles sharing a face with normal vector $\widetilde v$, only one of which is in $T$, which is a reformulation of \eqref{entile}. The rewriting \eqref{entile2} corresponds to considering a particle 
interaction between vertices of the multigrid, where only pairs of first neighbours 
(with respect to the natural graph structure of the multigrid given by its $1$-skeleton) 
are taken into account.

\subsubsection{Rescaled energy functionals and $\Gamma$-convergence statement}
We now recall the definition of the functionals that allow us to formulate our $\Gamma$-convergence result:
\begin{itemize}
    \item for each $N\in\mathbb N$ we define 
    \[
        \mathcal F_N(E):=
        \begin{cases}
            \displaystyle N^{-\frac{d-1}{d}}\mathcal{E}_{W}(T) & 
            \text{if $T:=N^{\frac1d}E$ is a disjoint union of $N$ tiles from $\mathcal T$},\\
            +\infty & \text{otherwise}.
        \end{cases}
    \]
    \item As a limit of $\mathcal F_N$ we find the desired perimeter functional:
    \begin{subequations}\label{pvnew}
        \begin{equation}
            P_{W}(E)=\begin{cases}\int_{\partial^*E} \phi_{W}(\nu_E(x)) d\mathcal H^{d-1}(x)
            &\text{if $E$ is a finite-perimeter set},\\ +\infty &\text{else.}\end{cases}
        \end{equation}
    where, with $A=\mathbb G \widetilde{\mathbb G}^T$ as in Lemma \ref{bda}, $\Lambda_J$ as in \eqref{deflambdaj}, and $U_{\Lambda_{J'_v}}$ as in formula \eqref{uvperp}, we have
        \begin{equation}\label{phivnew}
            \phi_{W}(\nu):=
            \frac{1}{\det A}\sum_{v\in\mathcal V}\frac{W(v)}{\mathcal H^{d-1}(U_{\Lambda_{J'_v}})}\langle \nu,Av\rangle_+.
        \end{equation}
    \end{subequations}
\end{itemize}

\medskip

Our main result then is the following.
\begin{theorem}\label{thmqc}
    The functionals $\mathcal F_N$ $\Gamma$-converge, 
    with respect to the $L^1_{loc}$ topology, to the functional $P_{W}$. 
    In particular, the minimizers $\overline{X}_N$ of $\mathcal E_{W}$ from \eqref{entile2},
    amongst $N$-tile configurations converge, 
    in a weak sense, up to rescaling, to a finite-perimeter set $E$ that minimizes $P_{W}$.
\end{theorem}

\subsection{Density of subtilings}
We now perform a “statistical analysis” of the tiling, counting how many tiles of every kind appear on average in a fixed portion of the space, proving a formula for the asymptotic density of each subtiling. We then apply these estimates to prove the main result of this section, namely Proposition \ref{prop:L1corresp}, which gives a correspondence between convergence of a sequence in the primal space and convergence of suitable sublattices in the dual space. Given the whole set of tiles $\T$ and $J\subset\G$, $\sharp J=d$, we denote by $\T_J$ the family (subtiling) composed by all tiles corresponding to points of the sublattice $\Lambda_J$, namely
\[
\T_J=\{P(x):x\in\Lambda_J\}.
\]
 We denote by $T^J=\bigcup_{T\in\T_J}T$ the corresponding union.
Given a subset $X_N\subset \mathcal X$ and the associated union of tiles $T_N=\bigcup_{x\in X_N}P(x)$, we define $X_N^J:=X_N\cap\Lambda_J$ and $T_N^J:=T_N \cap T^J$.

\begin{lemma}[Criterion for measure convergence]\label{lemma:L1} Let $E,E_N\subset\R^n$ be sets of finite measure. Then $E_N\to E$ in measure if and only if both of the following hold:
\begin{itemize}
  \item[(i)] $ |E_N|\to|E|$;
  \item[(ii)] $|E_N\cap B|\to |E\cap B|$  for every ball $B$.
\end{itemize}
\end{lemma}

\begin{proof}
We only prove that $(i)$ and $(ii)$ imply the convergence, because the other implication is trivial.
Assumption $(ii)$ implies the convergence of the $L^1$-$L^\infty$ pairing $\langle \1_{E_N},\1_F\rangle\to\langle \1_E,\1_F\rangle$ when $F$ is a ball. Thanks to Vitali's covering theorem, and using also assumption $(i)$, we extend this convergence to any measurable $F$. Moreover, given any $L^\infty$ nonnegative function $g$, we can write $g=\sum_{i=1}^\infty \lambda_i\1_{F_i}$ for some measurable sets $F_i$ and $\sum_{i=1}^\infty \lambda_i<\infty$. As a consequence we obtain $\langle\1_{E_N},g\rangle\to \langle \1_E,g\rangle$ for every $g\in L^\infty$, that is $\1_{E_N}\rightharpoonup\1_E$ weakly in $L^1$.

Using that $\1_{E_N\Delta E}=\1_{E_N}+\1_{E}-2\1_{E_N}\1_E$, we obtain that $\1_{E_N\Delta E}\rightharpoonup 0$ weakly in $L^1$ as well. Since the functions $\1_{E_N\Delta E}$ are positive, this implies strong convergence to $0$, which is equivalent to $|E_N\Delta E|\to 0$.
\end{proof}

The usefulness of the previous criterion stems from the fact that in order to prove convergence in measure, it is sufficient to prove an estimate for the volume on balls, without regard to where the set is located inside the balls. Using the decomposition of the multigrid in sublattices we will be able to precisely estimate the number of points of a given configuration inside balls, and using the bounded distortion property between primal and dual space we will transfer this information back and forth.

\begin{remark}
    We also mention the following result due to Visintin \cite{Vis}, that we could have used in the conclusion of the previous lemma: if a sequence $u_N\in L^1(\R^n,\R^m)$ is converging weakly to $u$, and if $u(x)$ is an extremal point of the closed convex hull $\overline{\mathrm{co}}(\{u_N(x)\}_N)$ for a.e. $x$, then $u_N\to u$ strongly in $L^1$. In our case the assumption on the convex hull is verified since all functions are characteristic functions.
\end{remark}

Recall the definition of the matrices $\mathbb{G}, \mathbb{\widetilde G}$ and $\mathbb{G}_J,\mathbb{\widetilde G}_J$ given below \eqref{eq:CauchyBinet}, and of the affine map $A$ with linear part $\mathbb G \mathbb{\widetilde G}^T$ given in \eqref{eq:A}.

\begin{lemma}[Density of subtilings]\label{lemma:density}
    For every ball $B_R(x)$, 
    \[
    |T^J\cap B_R(x)|=|B_R(x)|\rho_J+O(R^{d-1})\quad\text{as $R\to\infty$}
    \]
    where
    \begin{equation}\label{rhoj}
    \rho_J=\frac{|\det(\mathbb{G}_J)|\,|\det( \mathbb{\widetilde G}_J)|}{|\det (\mathbb G\mathbb{\widetilde G}^T)|}\stackrel{\eqref{detcond}}{=}\frac{\det(\mathbb{G}_J)\,\det( \mathbb{\widetilde G}_J)}{\det (\mathbb G\mathbb{\widetilde G}^T)}.
    \end{equation}
    In particular the subtiling $\T_J$ has asymptotic density $\rho_J$. Moreover, for every measurable $F$ with finite measure, we have
    \[
    \lim_{\lambda\to \infty} \frac{|T^J\cap \lambda F|}{|\lambda F|}=\rho_J.
    \]
\end{lemma}

\begin{proof}
Let $D$ be the diameter of the tiles in $\T_J$. We have that
    \[
    |T^J\cap B_R|=\sum_{\substack{T\in \T_J\\ T\cap B_{R-D}\neq \emptyset}} |T\cap B_R|+\sum_{\substack{T\in \T_J\\ T\cap B_{R-D}= \emptyset}} |T\cap B_R|.
    \]
The second term is bounded by $|B_R|-|B_{R-D}|=O(R^{d-1})$. In the first term, every tile is entirely contained in $B_R$ and therefore contributes with $|T|=|\det(\widetilde g: g\in J)|=(\Pi_{g\in J}|g|)\det \mathbb{\widetilde G}_J$ to the sum. We thus obtain
\[
|T^J\cap B_R|=(\Pi_{g\in J}|g|)\, \det \mathbb{\widetilde G}_J \,\sharp\{x\in\Lambda_J: P(x)\cap B_{R-D}\neq \emptyset\}+O(R^{d-1})
\]
To estimate the cardinality we make use of the bounded distortion property given by Lemma \ref{bda}. It follows that
\begin{align*}
\sharp\{x\in\Lambda_J: P(x)\cap B_{R-D}\neq \emptyset\}& =\sharp\{x\in\Lambda_J: x\in a^{-1}(B_{R-d})\}+O(R^{d-1})\\
&= |A^{-1}(B_R)| \frac{1}{|\det \Lambda_J|}+O(R^{d-1})\\
&=\frac{|B_R|}{|\det (\mathbb G\mathbb{\widetilde G}^T)|}    \frac{|\det \mathbb{G}_J|}{\Pi_{g\in J}|g|}+O(R^{d-1}).
\end{align*}
Here we have used that, thanks to the bounded distortion (Lemma \ref{bda}), the tile $P(x)$ associated to $x$ is at most at a bounded finite distance from $Ax$, and therefore the error obtained counting the points $x$ instead of the tiles $P(x)$ is of order $O(R^{d-1})$ (that corresponds to the number of points in a finite neighbourhood of $\partial B_R$).
Putting together the last two equations we obtain \eqref{rhoj}.

\medskip

The last statement about measurable sets follows from a rescaling and an application of Vitali's covering theorem.
\end{proof}

In the previous lemma we proved the existence of a density for every subtiling. In the following lemma we prove a more general statement: whenever a sequence $E_N$ is converging in measure to a set $E$, then every “restricted subtiling” $E_N\cap (N^{-1/d}T^J)$ is uniformly spread inside $E_N$, with the same density as the global one $\rho_J$.

\begin{lemma}\label{lemma:equisubs}
    Let $E$ be a measurable set of finite measure. If $E_N\to E$ in measure then 
    for every measurable set $F$ we have
    \[
    |N^{-1/d}T^J\cap E_N\cap F|\to \rho_J |E\cap F|,
    \]
    for $\rho_J$ as defined in \eqref{rhoj}.
\end{lemma}

\begin{proof}
From the convergence in measure we deduce that $|E_N\cap(E\cap F)|\to |E\cap F|$
for any $F$. Moreover
\begin{align*}
|E_N\cap (E\cap F)|&=\sum_J |N^{-1/d}T^J\cap E_N\cap (E\cap F)| \\
&\leq \sum_J |N^{-1/d} T^J\cap (E\cap F)|
\end{align*}
and the last term converges to $\sum_J \rho_J |E\cap F|=|E\cap F|$ by Lemma \ref{lemma:density}. Since the previous inequality holds term by term, we obtain that every single term has to converge to its upper bound, that is
\[
|N^{-1/d}T^J\cap E_N\cap (E\cap F)|\to \rho_J |E\cap F|.
\]
Since the limit of $|N^{-1/d}T^J\cap E_N\cap (E\cap F)|$ and of $|N^{-1/d}T^J\cap E_N\cap F|$ is the same, the conclusion follows.
\end{proof}

\begin{proposition}[$L^1$-correspondence between primal and dual]\label{prop:L1corresp}
Let $T_N=\bigcup_{x\in X_N} P(x)$ be a sequence of tiled sets, and suppose that $N^{-1/d}T_N\to E$ in measure for some set $E$ of finite measure. Then, setting
\begin{equation}\label{eq:ENJ}
 E_N^J:=\bigcup_{x\in X_N\cap \Lambda_J} (x+U_{\Lambda_J}),
\end{equation}
we have that $ N^{-1/d}E_N^J\to A^{-1}E$ in measure for every $J\subset\G$, $\sharp J=d$.
\end{proposition}

\begin{proof}
    We make use of the characterization of convergence in measure given by Lemma \ref{lemma:L1}. The aim is thus to prove that, for every ball $B$, 
    \begin{equation}\label{eq:subconv}
    | N^{-1/d} E_N^J\cap B|\to |A^{-1}E\cap B|\qquad\text{as $N\to\infty$}.
    \end{equation}
    We first rewrite the left-hand side in terms of $T_N^J$, using the bounded distortion property. We have
    \begin{align*}
    |N^{-1/d} E_N^J\cap B|&=N^{-1}|E_N^J\cap N^{1/d}B|\\
    &=N^{-1}\left(|\det \Lambda_J|\, \sharp\{x\in X_N\cap \Lambda_J: P(x)\subset N^{1/d}AB\}+O(N^{\frac{d-1}{d}})\right)\\
    &= N^{-1} |\det \Lambda_J| \,\sharp\{x\in X_N\cap \Lambda_J:N^{-1/d}P(x)\subset AB\}+O(N^{-1/d})\\
    &=\frac{|\det \Lambda_J|}{|\det(\widetilde g:\ g\in J)|} |N^{-1/d} T_N^J\cap AB|+O(N^{-1/d}),
    \end{align*}
    which converges to $\frac{|\det \Lambda_J|}{|\det(\widetilde g:\ g\in J)|} |E\cap AB|=\frac{|\det\Lambda_J|}{|\det(\widetilde g:\ g\in J)|} |\det A| \, |A^{-1}E\cap B|$ by Lemma \ref{lemma:equisubs}. Recalling \eqref{eq:det} and the definition of $\rho_J$ (see \eqref{rhoj}) we obtain
\[
|N^{-1/d}E_N^J\cap B|\to \frac{|\det \Lambda_J|}{|\det(\widetilde g:\ g\in J)|}|\det A|\, \rho_J|A^{-1}E\cap B|=|A^{-1}E\cap B|.
\]
We have thus proved \eqref{eq:subconv}, and by Lemma \ref{lemma:L1} we reach our conclusion.
\end{proof}

\subsection{Compactness}
Compactness of sequences of tile unions $T_N$ with equibounded $\F_N$-energy as defined in \eqref{eq:FNtile} is a direct consequence of compactness of finite perimeter sets, see \cite[Theorem~3.39]{AFP}. Indeed the assumption that $W(v)>0$ when $v$ is a normal to a face in the tiling directly implies that the functional $\F_N$ bounds the perimeter of $T_N$.

\subsection{Liminf inequality}\label{sec:glinfqc}
In order to prove the $\liminf$ inequality we will analyze every sublattice $\Lambda_J$ separately, and for each we will prove the lower bound \eqref{liminfj} concerning the energy of the bonds in a fixed direction $v\in\mathcal{V}$. To put these bounds together and prove the full $\liminf$ inequality, we will need the following combinatorial lemma. Let us start by defining
the edge-perimeter sets $\mathrm{EP}_v, \mathrm{EP}_{J,v}$ as follows: for $S\subset \Lambda_J$ we define 
\begin{equation}\label{notmb1}
\mathrm{EP}_{J,v}(S):=\left\{\{x,y\}\subset \Lambda_J\left|\begin{array}{c} x-y\text{ is the generator of }\Lambda_J\text{ parallel to }v,\\ x\in S, y\notin S,\end{array}\right.\right\},
\end{equation}
and for $X\subset \mathcal X$ we let 
\begin{equation}\label{notmb2}\mathrm{EP}_v(X):=\left\{\{x,y\}\subset \mathcal X:\ P(x),
P(y)\text{ share a face, } x\in X, y\notin X, x-y\parallel v\right\}.
\end{equation}
Recall also the definition of $J'_v$ given under \eqref{lines}.
\begin{lemma}\label{bondcount}
Let $X\subset \mathcal X$ be a finite set and let $v\in\mathcal V$. Then 
\begin{equation}\label{edgebound}
    \sum_{J\subset \mathcal G: \sharp J=d, J'_v\subset J} \sharp\ \mathrm{EP}_{J,v}(X\cap \Lambda_J) \le (\sharp \mathcal G-d+1)\ \sharp \mathrm{EP}_v(X).
\end{equation}
\end{lemma}
\begin{proof}
    Consider a multigrid direction $v\in\mathcal V$ and note that there are $\sharp \mathcal G -d +1$ distinct choices of $J$ such that for some $\lambda_J>0$ the vector $\lambda_J v$ is a generator of $\Lambda_J$. These $J$ are obtained by adding any new vector from $\mathcal G$ to the $(d-1)$-ple $J'_v$, associated to $v$ as in the definition of $\mathcal V$. 
    \medskip
    
    Now fix a multigrid line $\ell:=\mathbb R v + c$ parallel to $v$. For each pair $\{x,y\}\in \mathrm{EP}_v(X)$ contained in  $\ell$, and each $d$-ple $J\in \mathcal G$ of the form $J=J'_v\cup \{w\}$, there exists exactly one pair $\{x',y'\}\subset \Lambda_J$ with $x'-y'=\lambda_Jv$ and satisfying the open segment inclusion $(x,y)\subset (x',y')$. We define $\phi_{v,w}(\{x,y\}):=\{x',y'\}$ in this case, obtaining a map 
    \begin{equation}\label{phivw}
    \phi_{v,w}:\mathrm{EP}_v(X)\cap\{\{x,y\}:\ x,y\in \ell\}\to \{\{x',y'\}\subset \ell\cap \Lambda_J: x'-y'=\lambda_J v\}.
    \end{equation}
    We claim that the image of $\phi_{v,w}$ contains $\mathrm{EP}_{J,v}(X\cap \Lambda_J)\cap\{\{x',y'\}:\ x',y'\in\ell\}$. Indeed, let $\{x',y'\}$ belong to the latter set, and let's say that $x'\in X\cap \Lambda_J$ and $y'=x'+\lambda_J v\notin X\cap\Lambda_J$. Then $[x',y']$ is subdivided by the multigrid points $\mathcal X$ into a concatenation of segments $[x,y]\subset \ell$ such that the first segment has starting point in $X$ and the last one has end point outside $X$. Therefore there exists at least one segment in the concatenation that has one end in $X$ and the other outside $X$, i.e. it is an element of $\mathcal E(X)$. It follows from the definition of $\phi_{v,w}$ that $\phi_{v,w}(\{x,y\})=\{x',y'\}$ in this case, proving the claim. More precisely,
    \begin{multline*}
    \phi_{v,w}\left(\mathrm{EP}_v(X)\cap \{\{x,y\}:\ x,y\in\ell\}\right)\\
    \supset \left(\{\{x',y'\}:\ x',y' \in \ell\} \cap \mathrm{EP}_{J_v\cup\{w\},v}X\cap \Lambda_{J_v\cup\{w\}})\right).
    \end{multline*}
    By taking a union over $w\in\mathcal G\setminus J_v$, the above claim implies that the multimap given by $\phi_v(\{x,y\}):=\{\phi_{v,w}(\{x,y\}):\ w\in\mathcal G\setminus J_v\}$ satisfies 
    \begin{equation}\label{phivimage}
    \phi_v\left(\mathrm{EP}_v(X)\cap \{\{x,y\}:\ x,y\in\ell\}\right)\supset \bigcup_{\substack{J\subset \mathcal G\\\sharp J=d\\ J_v\subset J}}\left( \{\{x',y'\}:\ x',y' \in \ell\} \cap \mathrm{EP}_{J,v}(X\cap \Lambda_J)\right).
    \end{equation}
    Taking the union of \eqref{phivimage} over the multigrid lines of the form $\ell=\mathbb R v +c$, and observing that $\phi_v$ is actually $(\sharp \mathcal G-d+1)$-to-one because distinct $\Lambda_J$'s are disjoint, the bound \eqref{edgebound} follows.
\end{proof}
\medskip
{\bf Proof of the liminf inequality}

We consider $T_N$ which is a union of $N$ tiles from $\mathcal T$, 
and such that $N^{-\frac1d}T_N\to E$ in $L^1$ as $N\to\infty$. The corresponding $X_N\subset \mathcal X$ in the dual space can be partitioned along the multigrid lattices, into the subsets $X_N^J:=X_N\cap \Lambda_J$. 

\medskip

Recall the definition of the auxiliary sets $E^J_N:=\bigcup_{x\in X_N^J}(x+U_{\Lambda_J})$, first introduced in \eqref{eq:ENJ}. By Proposition \ref{prop:L1corresp} we have the $L^1$-convergence
\begin{equation}\label{ejconv}
N^{-\frac1d} E^J_N\to A^{-1}E,
\end{equation}
and we may then use the same setup as in the lattice case, Section \ref{latliminf}, to obtain
\begin{align}
\liminf_{N\to\infty}N^{-\frac{d-1}d} \sharp\mathrm{EP}_{J,v}(X_N^J)&\ge \int_{\partial^*(A^{-1}E)}\frac{1}{\det \Lambda_J} \langle \nu, \lambda_{v,J}v \rangle_+
d\mathcal H^{d-1}\nonumber\\
&=\int_{\partial^*(A^{-1}E)}\frac{1}{\H^{d-1}(U_{\Lambda_{J'_v}})}\langle\nu, v \rangle_+ d\H^{d-1}  ,\label{liminfj}
\end{align}
where we used equation \eqref{eq:lambdavJ}. Summing over all $J$ that contain $J'_v$ and using Lemma \ref{bondcount}, we obtain
\begin{equation}\label{eq:combinatorial}
N^{-\frac{d-1}{d}}\# EP_v(X_N)W(v)\geq \frac{1}{\#\G-d+1}\sum_{J:J_v'\subset J} N^{-\frac{d-1}{d}} \# EP_{v,J}(X_N^J)W(v).
\end{equation}
Taking now the $\liminf$ of the last sum and using \eqref{liminfj} we obtain the sum of $\sharp \G -d + 1=\sharp\{J\subset\mathcal G: \sharp J=d,\ J'_v\subset J\}$ equal terms, which cancels the factor $\tfrac{1}{\#\G-d+1}$. We are left with:
\[
    \liminf_{N\to\infty} N^{-\frac{d-1}{d}}\# EP_v(X_N)W(v)\geq \int_{\partial^*(A^{-1}E)} \frac{1}{\H^{d-1}(U_{\Lambda_{J'_v}})}W(v)\langle \nu, v \rangle_+d\H^{d-1}.
\]
Finally we also sum among all $v\in \mathcal{V}$, using that by definition (see \eqref{entile2})
\[
    \mathcal E_{W}(T_N)= \mathcal E_{W}(X_N)=\sum_{v\in\mathcal V} W(v)\sharp\mathrm{EP}_v(X_N).
\]
We obtain
\begin{equation}\label{glinfqc}
    \liminf_{N\to\infty}N^{-\frac{d-1}d}\mathcal E_{W}(T_N)\geq \sum_{v\in\mathcal{V}}\int\limits_{\partial^*(A^{-1}E)}\frac{1}{\H^{d-1}(U_{\Lambda_{J'_v}})} W(v) \langle \nu,v \rangle_+ d\H^{d-1}.
\end{equation}
Finally, using the change of variables formula given by Proposition \ref{prop:equivariance} for $M=A$ and with $V$ replaced by $W\circ A^{-1}$, we obtain exactly the $\liminf$ inequality involving the functional defined by \eqref{pvnew} and \eqref{phivnew}.

\subsection{Limsup inequality}
The proof follows the overall strategy from Section \ref{sec:glimsup1}, with a few additions.
The statement that we prove, analogous to the one proved in Section \ref{sec:glimsup1}, is the following.

\medskip

{\it For a given measurable set $E$ there exists $N_0\in\mathbb N$ depending only on 
$E, \mathcal T, \mathcal X$ such that for $N\in\mathbb N, N\ge N_0$, 
there exists $T_N\subset \mathbb R^d$ which is the union of $N$ tiles from $\mathcal T$, such that $-N^{\frac{d-1}{d}} \mathcal E(T_N)\to P_{W}(E)$
and $N^{-\frac1d}T_N\to E$ in $L^1$.}

\medskip

{\bf Strategy comparison to the crystal case:}

\medskip

The main difference to the case of lattices treated in Section \ref{sec:glimsup1} 
is that the sets $T_N$ do not have volume proportional to $N$, as the tiles from $\mathcal T$ 
do not all have the same volume. On the other hand, the set $\mathcal X$ in the dual space 
is a union of lattices, due to the multigrid construction. 
This implies sharp volume bounds for each lattice, analogous to the ones from Section \ref{sec:glimsup1}.

\medskip

Due to the above, we set up our desired approximation via steps analogous to Steps 1-4 
from Section \ref{sec:glimsup1}, extended to the multigrid setting. 
This produces sets of tile centers $X_N$, 
and we use the Bounded Distortion Lemma \ref{bda} to determine distortion bounds in an extra step 5, 
in order to show that our final tile set $T_N$ dual to $X_N$ satisfies $N^{-\frac1d}T_N\to E$ in $L^1$.

\medskip

{\bf In what follows, we will describe the main changes required for the quasicrystal proof, 
referring to the steps from Section \ref{sec:glimsup1} for details.}

\medskip

{\bf Step 0.} {\it Reduction to the case of $A=\mathbb G\mathbb{\widetilde G}=Id$ in Lemma \ref{bda}.} 
By the affine invariance of our functionals, we can rescale the multigrid by 
$(\mathbb G\mathbb{\widetilde G})^{-1}$ and reduce to the above mentioned case. From now on we thus assume 
that $\mathbb G\mathbb{\widetilde G}=Id$. This allows to simplify the notation and focus on more essential ideas.

\medskip

{\bf Step 1.} {\it Approximating $E$ by polyhedral steps of volume $1$.} 
This step is exactly the same as in Section \ref{sec:glimsup1}.

\medskip

{\bf Step 2.} {\it Approximation of polyhedral sets by discrete sets from the multigrid.} 
Note that $\mathcal X$ is the disjoint union of the lattices 
$\{\Lambda_J:\ J\subset \mathcal G, \sharp J=d\}$, 
in which $\Lambda_J$ has density $1/|\det \Lambda_J|$. 
We then consider a rescaling of $E$ (now assumed to be polyhedral and of volume $1$) of cardinality close to $N$ 
by correcting for the density of $\mathcal X$, which is the sum of densities of $\Lambda_J$:
\begin{equation}\label{defyn*}
    Y_N:=\left(\frac{N^{\frac1d}}{{\rho_{\mathcal X}}^{\frac1d}}E\right)\cap \mathcal X,
    \quad \text{where }\rho_{\mathcal X}
    :=\sum_{J\subset \mathcal G:\ \sharp J=d}\frac{1}{|\det \Lambda_J|}.
\end{equation}
Condition \eqref{xeps} for the multigrid setting follows by applying
the bound for lattices to all the lattices $\Lambda_J$, with 
$J\subset\mathcal G, \sharp J=d$, and using the triangle inequality.
We obtain a version of \eqref{xeps} with a constant depending on $\mathcal X$ 
due to this:
\begin{equation}\label{xeps1}
    |\sharp Y_N - N|\le 
    C_{\mathcal X}\mathcal H^{d-1}(\partial E) N^{\frac{d-1}{d}} 
    \quad \text{for}\quad N\ge N_{E,\mathcal X}.
\end{equation}
As a consequence of \eqref{xeps1}, we can continue and possibly 
add or remove a cluster of multigrid points 
obtaining from $Y_N$ a set $X_N$ satisfying \eqref{xnenrgy} 
with constants depending on $\mathcal X$ now, concluding this step.

\medskip

{\bf Steps 3-4.} {\it Approximating $\mathcal E_{W}(T_N)$ by the contribution of interiors of faces.} 
Our decomposition of the energy will now use the rails from Section \ref{sec:rails} 
to decompose the energy $\mathcal E_{W}$ into contributions coming from all families of rails, 
which replaces the role of the contributions from parallel rays in Step 3 of Section \ref{sec:glimsup1}. 
Note that if $\mathbb G\widetilde{\mathbb G}^T=Id$ then rail directions are $v\in \mathcal V$. 
Thus $\mathcal H^{d-1}(U_{v^\perp})$ from Step 3 of Section \ref{sec:glimsup1} is now replaced by $\mathcal H^{d-1}(U_{\Lambda_{J'_v}})$ (see the paragraph preceding \eqref{uvperp} for this notation). 
With these substitutions we reach a quasicrystal analogue of 
\eqref{bdminko}, \eqref{bdint1} and \eqref{bdint2}, 
with the only difference being that $C_d'$ now depends on the multigrid data $\mathcal T$,
and that in the quasicrystal case $P^v(N^{\frac1d}E)$ has to be defined as the contribution
of $v\in\mathcal V$ to $P_{W}$, i.e. the contribution of the summand in \eqref{phivnew} corresponding
to $v$ only. With this we conclude Steps 3-4. We note here a weakened simplified version of 
the final estimate analogous to \eqref{bdint3}, in which $C_1, C_2$ only depend on $\mathcal T, \mathcal X, E, W$ (see \eqref{bdint3} for a more precise dependence in the lattice case):
\begin{eqnarray}
0&\ge&P_{W}(E) + \mathcal E_{W}(X_N)\nonumber\\
&\ge&C_1 N^{-\frac{d-1}{d^2}}\mathcal H^{d-1}(\partial E) - C_2 N^{-\frac{1}{d}}\sum_{P\text{ face of }E}\mathcal H^{d-2}(\partial P).\label{glimsupqc}
\end{eqnarray}
{\bf Step 5.} {\it Passage from the dual to the primal space.}
Note that we have $\mathcal E_{W}(X_N)=\mathcal E_{W}(T_N)$, 
and the volume of tiles in $\mathcal T$ which correspond to 
the symmetric difference between $Y_N\ \Delta\ X_N$ 
is of lower order $O(N^{\frac{d-1}d})=o(N)$ due to \eqref{xeps1}. 
Thus in order to check that $\left(\frac{N^{\frac1d}}{{\rho_{\mathcal X}}^{\frac1d}}\right)^{-1}Y_N\to E$ 
in $L^1$ it suffices to prove the following:
\begin{lemma}\label{volconv}
    If $Y_N, \rho_{\mathcal X}$ are as in \eqref{defyn*}, there exists $N_0\in\mathbb N$
    such that if $N\ge N_0$ and if $\mathbb G \widetilde{\mathbb G}^T=Id$, 
    then for the corresponding unions of tiles $S_N$ which correspond to points in $Y_N$, there holds
    \begin{equation}\label{ynconv}
        \left|S_N\ \Delta\ \left(\frac{N^{\frac1d}}{{\rho_{\mathcal X}}^{\frac1d}}E\right)\right| \le C\ N^{\frac{d-1}d}\ \mathcal H^{d-1}(\partial E) .
    \end{equation}
\end{lemma}
\begin{remark}
    Although not used here, it is interesting to note that by the Laczkovic characterization of bounded displacement \cite{Lac92}, 
    the BD property shown in Lemma \ref{bda} is actually 
    equivalent to a property of the type \eqref{ynconv}. 
    The original work of \cite{Lac92} treats the case of cubic tiles, 
    while substitution tilings are treated in \cite{Sol11} based on the same basic idea, 
    and probably the multigrid analogue can be treated similarly too, based on Hall's marriage lemma. 
    We only prove the implication from Lemma \ref{bda} to \eqref{ynconv} here.
\end{remark}
\begin{proof}[Proof of Lemma \ref{volconv}:]
    We denote $E_N:=\frac{N^{\frac1d}}{{\rho_{\mathcal X}}^{\frac1d}}E$ and note that, again with
    the notation of Lemma \ref{bda}, there holds
    \[
        S_N=\bigcup_{y\in \widetilde{Y}_N}(\text{tile in $\mathcal T$ with center }y)
        ,\quad \text{where}\quad \widetilde Y_N:=\phi(Y_N)=\phi(\mathcal X\cap E_N).
    \]
    We note that if $C_{diam}$ is the maximum diameter of a tile from
    $\mathcal T$, and $C_{BD}:=\sup_{x\in\mathcal X}|\psi(x) - Ax|$ with the notation of Lemma \ref{bda}, then 
    \[
        \partial S_N\subset\{x\in\R^d:\ \mathrm{dist}(x,\partial E_N)\}<C_{BD}+C_{diam}.
    \]
    As the constants $C_{diam}, C_{BD}$ depend only on $\mathcal G, \widetilde{\mathcal G}$,
    it follows that for $N$ sufficiently large (depending on $\mathcal G$,
    $\widetilde{\mathcal G}$ and $E$) such that the bound \eqref{ynconv} holds, with $C$ 
    depending only on $\mathcal G, \widetilde{\mathcal G}$.
\end{proof}

\subsubsection{Conclusion of proof of Theorem \ref{thmqc}}
Recall that we have already proved the $\Gamma$-liminf inequality corresponding to the statement of Theorem \ref{thmqc} in Section \ref{sec:glinfqc}.

\medskip

For the $\Gamma$-limsup part, we use the $X_N$ as recovery sequence, and we showed in Step 5 above that $\sharp(Y_N\Delta X_N)=O(N^{\frac{d-1}{d}})$. The bound \eqref{ynconv} implies in particular that $\left(\frac{N^{\frac1d}}{{\rho_{\mathcal X}}^{\frac1d}}\right)S_N\to E$ in $L^1$, and by taking the limit in \eqref{glimsupqc} we find that $\lim_{N\to\infty}\mathcal{E}_{W}(X_N^*)=P_{W}(E)$, as desired. This concludes the proof of Theorem \ref{thmqc}.
\hfill $\square$

\section{Building complex Wulff shapes from simpler ones}\label{sec:wshapes}

In this section we briefly discuss the relation between the Wulff shapes of two potentials and the Wulff shape of their sum. We then show in Proposition \ref{prop:equivsymm} that for all our purposes it is sufficient to consider symmetric potentials, that is those satisfying $W(w)=W(-w)$ for every $w$, since the energy is invariant under symmetrization of $W$.

\subsection{Basic properties}
We first relate the functional
\[
P_\phi(E):=\int_{\partial^*E}\phi(\nu_E)d\mathcal H^{d-1}
\]
to equivalent formulations. Instead of $\phi=\phi_{W}$ as before, we consider first the case of a general convex positively $1$-homogeneous $\phi:\R^d\to [0,+\infty)$. We recall (see \cite{Tay78} for a proof) that the Wulff shape (i.e. a shape homothetic to any minimizer of $\int_{\partial^* E} \phi(\nu_E)d\mathcal H^{d-1}$ under fixed volume constraint) is then given by the subdifferential $\partial^-\phi(0)$ at the origin:
\begin{equation}\label{wulff-def}
\mathcal{W}_\phi:=\partial^-\phi(0)=\{x\in\R^d:\ \forall y\in\R^d,\ \langle x,y\rangle \le \phi(y)\}.
\end{equation}
\begin{remark}
In case $\mathcal{W}_\phi$ as defined in \eqref{wulff-def} has zero volume, 
then the isoperimetric problem is not well posed, however the definition
\eqref{wulff-def} still is the natural extension by continuity, 
of the definition for the non-degenerate case. Note that the Wulff shape $\mathcal{W}_\phi$ obtained in this case is the optimal shape for the case discussed in Section \ref{subsec:lowerdim}.
\end{remark}
For any nonnegative $\phi$ we have that $\mathcal{W}_\phi$ is compact, convex and contains the origin. Moreover, the origin is an interior point of $\mathcal{W}_\phi$ if $\mathcal{W}_\phi$ has nonzero volume. Viceversa if $K$ is a convex set containing the origin then we can define
\[
\phi^{K}(x):=\sup\{\langle x,y\rangle:\ y\in K\},
\]
and note that then $\phi^{\mathcal{W}_\phi}=\phi$ and $\mathcal{W}_{\phi^K}=K$. If $\phi(x)=\phi(-x)$ and $\phi(x)>0$ for $x\neq 0$, then $\phi$ defines a norm on $\R^d$, which is dual to the norm whose unit ball is $K$, and which has the polar body $K^*$ of $K$ as unit ball.
Common terminologies refer to $\phi^{K}$ as the \emph{surface tension function of $K$} or as the \emph{support function of $K$}.

\subsection{Sums and differences of potentials and of Wulff shapes}\label{sec:sumdiff}
Given two potentials $\phi_1$ and $\phi_2$ a natural question to ask is what is the Wulff shape associated to $\phi_1+\phi_2$. As shown in \cite[Thm. 1.7.5]{Sch14}, the answer is the Minkowski sum of the Wulff shapes of $\phi_1$ and $\phi_2$, i.e. we have that
\begin{equation}\label{wulffsum}
 \mathcal{W}_{\phi_1+\phi_2}=\mathcal{W}_{\phi_1}+\mathcal{W}_{\phi_2}.
\end{equation}
We note that a special case in which \eqref{wulffsum} gives important information, is that of $\phi(\nu):=|\langle \nu, v\rangle|$ for some $v\in \R^d\setminus\{0\}$. In this case we have $\mathcal{W}_\phi=[-v,v]$, according to \eqref{wulff-def}. By \eqref{wulffsum} it follows that for $\phi(\nu)=\sum_{v\in \mathcal N}W(v)|\langle \nu,v\rangle|$ with $W$ positive, we have that $\mathcal{W}_\phi=\sum_{v\in\mathcal N}W(v)[-v,v]$, in other words $\mathcal{W}_\phi$ is a \emph{zonotope}, i.e. a Minkowski sum of segments. Equivalently, a zonotope is a convex polytope all of whose $2$-dimensional faces are centrally symmetric \cite[Thm. 3.5.1]{Sch14}. More generally, a \emph{zonoid} is a convex body which is the Hausdorff-distance-limit of zonotopes. A zonoid $Z$ has in general support function $\phi^Z$ representable as a superposition
\begin{equation}\label{zonoid}
\phi^Z(\nu)=\int|\langle \nu,v\rangle| d\mu_Z(v),
\end{equation}
in which $\mu_Z$ is a positive measure. Using $1$-homogeneity, we may further impose that $\mu_Z$ 
be supported on the unit sphere $\mathbb S^{d-1}\subset\R^d$, 
in which case $\mu_Z$ is uniquely determined by $Z$ (see \cite[Thm. 3.5.3]{Sch14}).

\medskip

Recall that Minkowski subtraction is defined by $A-B:=\bigcap_{x\in B}(A-b)=\mathbb R^d\setminus((\mathbb R^d\setminus A) + B)$, or $A-B$ can be characterized as the maximal convex set $C$ such that $C+B\subset A$. As a consequence, $\phi^{A-B}$ is the convexification of $\phi^A-\phi^B$, i.e. the largest positively $1$-homogeneous convex function bounded above by $ \phi^A-\phi^B$. As already noted e.g. in \cite[Lem. 1(b)]{dk2000} with a different notation, we then have that $\mathcal{W}_{\phi^A-\phi^B}=\mathcal{W}_{\phi^{A-B}}=A-B$. Therefore, for convex $1$-homogeneous $\phi_1,\phi_2$ such that $\phi_1,\phi_2, \phi_1-\phi_2\ge 0$, we also have the following expression via Minkowski subtraction: 
\begin{equation}\label{eq:wulffdiff}
\mathcal{W}_{\phi_1-\phi_2}=\mathcal{W}_{\phi_1} - \mathcal{W}_{\phi_2}.
\end{equation}
Note that if $\phi(x)=\phi(-x), j=1,2$ and $\mathcal{W}_\phi$ is $2$-dimensional, then it is centrally symmetric, and is a zonoid. Centrally symmetric Wulff shapes of dimensions $\ge 3$ are generally not zonoids: see the examples below, in which $\phi$ is of the form \eqref{formphi} below.

\medskip

The important case for our work is that of $\phi$ of the special form
\begin{equation}\label{formphi}\phi(\nu)=\sum_{v\in\mathcal N}V(v)|\langle \nu,v\rangle|,
\end{equation}
for finite $\mathcal N\subset\mathbb R^d$, and where $V:\mathcal N\to\mathbb R$. We can always split the sum into two sums $\phi=\phi_+-\phi_-$ corresponding to the decomposition $\mathcal N=\mathcal N_-\cup\mathcal N_+$, so that $\mathrm{sign}(V(v))=\pm$ on $\mathcal N_\pm$, then by \eqref{eq:wulffdiff} $\mathcal{W}_\phi=\mathcal{W}_{\phi_+}-\mathcal{W}_{\phi_-}$. This shows that for $\phi$ as in \eqref{formphi}, the Wulff shape $\mathcal{W}_\phi$ is the Minkowski difference of two zonotopes: this is in general not a zonotope, as shown by a few examples below.

\medskip

{\bf Question: } Are all convex centrally symmetric polyhedra expressable as the difference of two zonotopes? (This is trivially true in dimension $2$, thus the question is really interesting for dimensions $3$ and higher.)

\begin{remark}\label{rmk:wulffdiff}
A useful observation (which we found in \cite[Thm. 1]{althoff}) which allows to express the difference of two zonotopes in an algorithmic simple way, is the following. For convex sets $A,B,C$ the Minkowski sum and difference satisfy $A-(B+C)=(A-B)-C$. Thus if $A$ is a convex set and a zonotope $Z$ is written as $Z=c+\sum_{i=1}^n[-v_i,v_i]$ with $c,v_1,\dots,v_n\in\mathbb R^d$, then we have, denoting $I_j=[-v_j,v_j]$,
\[
A-Z = \left(\cdots\left(\left((A-c)-I_1\right) -I_2\right)\cdots - I_n\right).
\]
To simplify the description further, note that $A-[-v,v]=(A+v)\cap(A-v)$.
\end{remark}

\begin{remark}
    A statement that we found in several instances in the convex analysis literature (see e.g. \cite[Cor. 3.5.7]{Sch14}, which is based on \cite{Sch70}), is that if $\mathcal{W}_\phi$ is a polytope with $\phi$ as in \eqref{formphi}, then it is necessarily a zonotope: this is false as shown by the examples below. The main correction required in \cite{Sch14} seems to be in the proof of \cite[Lem.~3.5.6]{Sch14}, when it is mentioned that two integrals over sets $B, B'$ with respect to a weight $\rho$ are equal: this can be false for general (possibly concentrating) $\rho$, such as $\rho$ with atoms on $B\setminus B'$, as is our case. It is important to emphasize that the above theorem and proof remain valid for regular convex bodies, and counterexamples are possible only the case of polytopes and non-smooth convex bodies.
\end{remark}

\begin{example}[Quasicrystal non-zonotope Wulff shapes] One of the main new ideas in \cite{IngSte89} was to allow negative weights, to produce new limiting isoperimetric shapes in quasicrystals. In particular, seemingly for the first time in \cite{IngSte89}, a possible choice of $\mathcal N$ and of a signed weight $V$ over $\mathcal N$ is given, such that if $\phi$ as in \eqref{formphi}, then $\mathcal{W}_\phi$ is a regular dodecahedron. As mentioned above, such $\mathcal{W}_\phi$ is not a zonotope however, importantly, it was observed in actual quasicrystalline materials. The choice of \cite{IngSte89} is to take $\mathcal N=\mathcal N_+\cup\mathcal N_-$ where $\mathcal N_-$ are the directions of vertices of a regular icosahedron centered at the origin, and $\mathcal N_+$ are the unit vectors $v$ pointing towards the midpoints of the sides of the same icosahedron. Then they set $V=-1$ on $\mathcal N_-$ and $V(v)=5/6$ on $\mathcal N_+$. Also icosahedral Wulff shapes can be obtained with $\phi$ as in \eqref{formphi}, see \cite[Fig. 6]{IngSte89}.
\end{example}

\begin{remark}Related to the above example, note that any choice $V=c_+1_{\mathcal N_+} - 1_{\mathcal N_-}$ with $c_+\in(\frac{3\varphi + 4/3}{\sqrt{2+\varphi}(3\varphi+1)},5/6]$ where $\varphi=\frac12(1+\sqrt5)$ also gives a dodecahedral $\mathcal{W}_\phi$. Here the constant $5/6$ is taken directly from \cite{IngSte89}, whereas the lower bound on $c_+$ follows by requiring $\phi_V>0$.

The allowed interval for $c_+$ is given by the conditions that $\phi>0$ and that only the faces with normals $\mathcal N_-$ give supporting planes of $\mathcal{W}_\phi$.

\end{remark}

\begin{example}[Crystalline non-zonotope Wulff shapes]\label{crystnozono}\hfill 

{\bf A. } It is worth mentioning that non-zonotope crystal shapes can in theory be formed, and are actually observed in nature, also in the presence of true lattice-like microscopic structure, such as in pyrite, which forms a (non-regular) dodecahedral Wulff shape, also known as a "pyritohedron", and having vertices
\[
(\pm1,\pm1,\pm1) \quad\text{and cyclic permutations of}\quad \left(0, \pm\frac32, \pm\frac34\right).
\]
This shape can be obtained with $\phi$ in the form \eqref{formphi}, for example with
\begin{eqnarray*}
 \mathcal N_+&:=&\{(\pm 1,0,0), (\pm 4, \pm 2, \pm 1) \text{ and cyclic permutations}\},\\
 \mathcal N_-&:=&\{(0,\pm2,\pm4)\text{ and cyclic permutations}\}.
\end{eqnarray*}
We can set $V_+(\pm e_j)=2/3$ and $V_+(v)=\frac{4}{21}$ for the remaining vertices in $\mathcal N_+$, and then the Wulff shape of $\phi_{V_+}$ is the zonohedron (recall that by definition a zonohedron is a $3$-dimensional zonotope) with set of edge vectors equal to those of the above pyritohedron. Note that $\mathcal{W}_{\phi_{V_+}}$ then has $12$ decagonal sides corresponding to pentagonal pyritohedron sides, as well as $30$ rectangular faces and $20$ hexagonal faces. If now $V_-$ is constant on $\mathcal N_-$, the Wulff shape of $\phi_{V_-}$ is a zonohedron with sides equal to the normals to pyritohedron faces. If the constant value of $V_-$ is large enough, the Minkowski subtraction of $\mathcal{W}_{\phi_{V_-}}$ from $\mathcal{W}_{\phi_{V_+}}$, which gives $\mathcal{W}_{\phi_{V_+}-\phi_{V_-}}$, has by Remark \ref{rmk:wulffdiff} the effect of removing suitable sides of the decagon facets from $\mathcal{W}_{\phi_{V_+}}$, and thus allows to obtain the pyritohedron as $\mathcal{W}_{\phi_{V}}=\mathcal{W}_{\phi_{V_+}-\phi_{V_-}}$.

{\bf B. } As a simpler situation realizable in $\mathbb Z^3$, and such that the Wulff shape $\mathcal{W}_\phi$ is an octahedron (again not a zonotope, thus not realizable with positive $V$), we can take $\phi$ as in \eqref{formphi} and $\mathcal N=\mathcal N_-\cup\mathcal N_+$ where 
\[\mathcal N_+=\left\{\text{cyclic permutations of }(0,\pm 1,\pm1)\right\}, \quad \mathcal N_-=\{\pm e_1,\pm e_2,\pm e_3\},
\]
and again we take  $V=V_+-V_-=c_+1_{\mathcal N_+} - 1_{\mathcal N_-}$. Note that $\mathcal N_+$ are the nearest-neighbors in a face-centered-cubic (FCC) lattice, and the Wulff shape corresponding to $V_+$ is a truncated cuboctahedron of sidelength $c_+2\sqrt 2$, while the Wulff shape corresponding to $V_-$ is a cube of sidelength $2$. Following the previous remark as well as Remark \ref{rmk:wulffdiff} we find that values $c_+\in\left(\frac14, \frac12\right]$, give an octahedral Wulff shape.
\end{example}

\subsection{\texorpdfstring{$\Gamma$}{Gamma}-limits for signed potentials \texorpdfstring{$V$}{V}}\label{sec:gammalimsigned}
While it is interesting to note that non-zonohedral Wulff shapes can occur, as mentioned in \cite{IngSte89} and explained in Section \ref{sec:sumdiff}, another relevant question in the framework of the current paper is whether these shapes can occur as $\Gamma$-limits of discrete perimeter functionals like in Theorems \ref{thm:main} and \ref{thmqc}. In this section we show that this is possible but only under additional hypotheses, without which the convergence result can be false. In particular the study at the discrete level presents new difficulties which are invisible at the continuum level.

\begin{example}[Admissible $\phi_V$, non-converging discrete $\mathcal P_V$]\label{negnbs}
Consider the case that, on the lattice $\mathbb Z^2$ and with the notation of Section \ref{sec:sumdiff}, $\mathcal N_+=\{(\pm 1,\pm 1)\}$ and $\mathcal N_-=\{\pm e_1, \pm e_2\}$. These choices give a nontrivial (square) Wulff shape for $\phi_V$ with the weights $V=c 1_{\mathcal N_+} - 1_{\mathcal N_-}$, provided $c> 1$, as easily seen by the method in \eqref{eq:wulffdiff}. 

On the other hand, note the following pathological example: say that for $N$ a square, $X_N=K_N\cap \mathrm{Span}_{\mathbb Z}(\mathcal N_+)$, where $K_N$ is a square of sidelength $\sqrt{2 N}$ with edges parallel to the vectors of $\mathcal N_+$ (thus $K_N$ is equal to a rescaling of the continuum Wulff shape corresponding to weight $\phi_V$ for $c>1$). In this case, in the ``bulk'' of $K_N$, all available $\mathcal N_-$-bonds participate to its $V$-perimeter. For large $N$ the number of negative contributions is $O(N)$ while possible positive contributions are $O(N^{\frac12})$, corresponding to bonds from the ``boundary layer'' of $K_N$. This means that the discrete $V$-weighted perimeter of $X_N$ is $- C N +O(N^{\frac12})$ for some constant $C$ depending on $V$. This holds even for $c>1$ very large, and exhibits an essential difference compared to the continuum perimeter $\phi_V$ (which as mentioned above, is positive for $c>1$). 

The above phenomenon can be exploited to prove that the $\Gamma$-convergence is also false for the above $V$: if to the $X_N$ from the above paragraph we add all the missing points of $K_{N-N_1}\cap \mathbb Z^2$ for some fixed large $N_1$, it is not hard to see that the so-obtained set $\widetilde X_{\widetilde N}$ has cardinality $\widetilde N\le 2N$ and discrete $V$-weighted perimeter bounded below by 
\[
-CN_1N^{\frac12} + C_1 N^{\frac12}\le -\widetilde C N^{\frac12}=-\frac{\widetilde C}{\sqrt 2}\widetilde N^{\frac12},
\]
for large enough $N_1$. For the above to hold, we need to choose $N_1$ depending on $V$ only, and thus we can fix it independently of $N$. Informally speaking, the negative contributions in a $N_1$-neighborhood of $\partial K_N$ can always compensate the positive contributions in a $\sqrt2$-neighborhood of $\partial K_N$. On the other hand, if $N_1$ is fixed then $\widetilde N^{-\frac12}E_{\widetilde N}(\widetilde X_N)\stackrel{L^1}{\to} K_0$, where $K_0$ is a cube of sidelenght $\sqrt2$ and sides parallel to vectors from $\mathcal N_+$, which is a rescaling the Wulff shape of $\phi_V$. The $\phi_V$-perimeter of $K_0$ is positive and the rescaled $V$-weighted perimeters of $\widetilde X_N$ are negative, showing that the $\Gamma$-convergence result of Theorem \ref{thm:main} is false for all our $V$ with $c>1$.\end{example}

Note that the pathological example \eqref{negnbs} can be modified to hold in any dimension and in a variety of situations. We formulate the following condition on $V$, depending on a parameter $\epsilon>0$:
\begin{equation}\label{condbepsilon}
    \text{For each }X\subset \mathcal L,\quad \mathcal F_V(X) \ge \epsilon \mathcal F_{1_{\mathcal N}}(X).
\end{equation} 
This condition forbids the pathology of Example \eqref{negnbs}, at least for the case of finite $\mathcal N$, and it may help for an extension of the $\Gamma$-convergence result of Theorem \ref{thm:main} to signed $V$, if satisfied. We leave this investigation for future work.

\medskip

Note that Example \eqref{negnbs} produces non-$\Gamma$-converging discrete energies mainly based on the structure of $\mathcal N_+, \mathcal N_-$ and not as much on the further degrees of freedom of the  weights $V_+,V_-$. In the general case this phenomenon is confirmed by the following Proposition \ref{connectedhyp}. We recall that if $\mathcal N\subset\mathbb R^d$ is a subset, we say that a set $A\subset \mathbb R^d$ is {\bf $\mathcal N$-connected}, if for any $x,y\in A$ there exists an $\mathcal N$-path in $A$ from $x$ to $y$, i.e. there exists $n\in\mathbb N$ and points $x=x_0,x_1,\dots,x_n=y \in A$ such that $x_i-x_{i-1}\in\mathcal N_+$ for all $i$. With this definition, we have the following:
\begin{proposition}\label{connectedhyp}
Let $V=V_+-V_-$ and $\mathcal N=\mathcal N_+\cup \mathcal N_-$ be as before. If $\inf V_+>0$ and $\mathcal N_-\cup \{0\}$ is finite and $\mathcal N_+$-connected, then there exists a constant $C=C_V>0$ such that \eqref{condbepsilon} holds for $\epsilon=1$ for $W=C_V V_+ -V_-$.
\end{proposition}
\begin{proof}
The proof is in a similar spirit to the one of Lemma \ref{lemma:combinatorial}, so we refer to that proof more details. 

\medskip To each pair of $\mathcal N$-neighbors $x,y\in\mathcal L$ we assign an $\mathcal N_+$-path in $\mathcal L$ from $x$ to $y$, denoted $p_{x,y}$. To do this, it is sufficient to arbitrarily choose $p_{0,v}$ for all $v\in\mathcal N_-$, define $p_{0,v}=\{0,v\}$ for $v\in\mathcal N_+$, then extend the definition by translation to all pairs of $\mathcal N$-neighbors $x,x+v$. Since $\mathcal N$ is finite, we find by a simple covering argument that there exists $C_1>0$ depending on our choice of paths $p_{0,v},v\in\mathcal N_-$ only, such that
\[
\max_{v\in\mathcal N_+}\max_{z\in\mathcal L} \sharp \{p_{x,y}:\ \{z,z+v\}\text{ is a step of }p_{x,y}\}\le C_1.
\]
If $X\subset\mathcal L$ and we have $x\in X, x+v\notin X, v\in\mathcal N$, then $p_{x,x+v}$ has at least one edge $\{x_{i-1},x_i\}$ such that $x_{i-1}\in X, x_i\notin X$. It follows that 
\begin{eqnarray*}
\lefteqn{C_1\ \sharp\{(x,y):\ y-x\in\mathcal N_+,\ x\in X, y\notin X\} }\\&\ge&\sharp \{(x,y):\ y-x\in\mathcal N, x\in X, y\notin X\}\\ 
&=& \mathcal F_{1_{\mathcal N}}(X).
\end{eqnarray*}
Thus if we choose $C_V=C_1 \inf_{v\in V_+}V(v)$ the thesis follows.
\end{proof}

\begin{remark}\label{nozono}
While the choice of constant $C_V$ from Proposition \ref{connectedhyp} which follows from the proof method is in general far from optimal, there exists such $C_V$ such that \eqref{condbepsilon} holds and the Wulff shape $\mathcal{W}_{W}$ for $W$ as in Proposition \ref{connectedhyp}, is not a zonotope. For example, part B of Example \ref{crystnozono} satisfies the hypotheses of Proposition  \ref{connectedhyp}. The ensued potential $C_VV_+ - V_-$ for very large $C_V>1$ creates as Wulff shape a truncated octahedron whose hexagonal sides have alternating very long and very short edges, and thus are not centrally symmetric, so that the corresponding $\mathcal{W}_{W}$ is not a zonotope. 
\end{remark}
\begin{remark}\label{qcnozono}
The same considerations as in the rest of this subsection, in particular the ones from Remark \ref{nozono}, can be applied to the case of quasicrystals under signed potentials $V$ with similar conclusions, but the condition from Proposition \ref{connectedhyp} can only justify some of the non-zonohedral Wulff shape limits corresponding to the $\phi_V$-Wulff shapes examples presented in \cite{IngSte89}, and not all of them. For a complete theory a sharper condition on $V$ for the extension of Theorems \ref{thm:main} and \ref{thmqc} to signed $V$ is needed. We plan to address the study of these sharp versions to future work.
\end{remark}

\subsection{Reduction to the symmetric case}
In this section we show that for the study of Wulff shapes, it is no restriction to assume that the support function $\phi$ is an even function.
The case of functionals 
\begin{equation}\label{phivgen}
\phi_V(x):= \sum_{v\in \mathcal N}V(v)\langle v,x\rangle_+,
\end{equation}
coming from a potential $V:\mathcal N\to \mathbb R$ where $\mathcal N\subset\mathbb R^d$ is finite or countable and $\sum_{v\in \mathcal N}|v||V(v)|<\infty$ can be reformulated in a form generalizing \eqref{zonoid}, in which instead of the absolute value we consider the positive part of the integrand, and we consider signed finite measures with at most countable support. Indeed, given a potential $V$ as above, we observe that
\begin{equation}\label{zonoid+}
\phi_V(\nu)=\int \langle \nu,v\rangle_+ d\mu(v)
\end{equation}
where $\mu=\sum_{v\in \N} V(v)\delta_v$.  

\medskip Note that our general class of $V$ includes the case of finite $\mathcal N$ and constant-sign $V$ from in \eqref{eq:phiV}. Also note that every $\phi$ in the form \eqref{zonoid} can be put in the form \eqref{zonoid+}, just considering the symmetrized measure $\mu=\mu_Z^{sym}=\tfrac12(\mu_Z+(-Id)_\# \mu_Z)$ (which does not change the value in \eqref{zonoid}).

\medskip As we now show, in \eqref{phivgen} we may always replace "$\langle v,x\rangle_+$" by "$|\langle v,x\rangle|$" and consider neighbor sets $\mathcal N=-\mathcal N$ and even $V$ only. Indeed, for every finite perimeter set $E$ in $\R^d$ we have
\begin{equation}\label{eq:equivP}
P_V(E)=P_{V^{sym}}(E)
\end{equation}
where $V^{sym}(v)=\tfrac12\big(V(v)+V(-v)\big)$. This equality is due to the area formula and to the fact that, for every direction $e\in \S^{d-1}$ that we fix, the number of times that the line $\R e$ “enters” $E$ is equal to the number of times it “exits” $E$. Note that \eqref{eq:equivP} could fail in the case of relative perimeter functionals $P_\phi(E,\Omega)$, where we compute the perimeter of $E$ relative to a fixed open subset $\Omega$ of $\R^d$, or when the potential $V$ is not a superposition of terms of type $\langle\nu,v\rangle_+$. However these cases are not considered here, and the two formulations \eqref{phivgen} and \eqref{zonoid+} are completely equivalent in our setting. For this reason we will assume from now on that $V$ (and thus $\phi_V$) is symmetric under reflection with respect to the origin.

\begin{proposition}[Reduction to the symmetric case]\label{prop:equivsymm}
For every finite measure $\mu$ on $\S^{d-1}$ define
\[
\phi_\mu(\nu)=\int_{\S^{d-1}} \langle \nu,v\rangle_+d\mu(v).
\]
Let $\widetilde\mu=(-Id)_\# \mu$. Then for every finite perimeter set $E\subset\R^d$ we have
\[
\int_{\partial^* E} \phi_\mu(\nu)d\H^{d-1}=\int_{\partial^* E} \phi_{\widetilde\mu}(\nu)d\H^{d-1}.
\]
\end{proposition}

\begin{proof}
We first prove the result for bounded smooth sets $E$ and then conclude by a density argument. Assume then that $E$ is a bounded, finite perimeter set with smooth boundary. Fix a direction $e\in \S^{d-1}$. We decompose $\partial E$ with respect to the direction $e$, defining 
\begin{align*}
\partial_e^+ E&=\{x\in \partial E: \langle \nu_E(x),e\rangle>0\}\\
\partial_e^- E&=\{x\in \partial E: \langle \nu_E(x),e\rangle<0\}\\
\partial_e^{\,0} E&=\{x\in \partial E: \langle \nu_E(x),e\rangle=0\}.
\end{align*}
Then by the area formula
\[
P(E)\geq \int_{\partial E} \langle \nu(x),\pm e\rangle_+ d\H^{d-1}(x)
=\int_{\pi_{e^\perp} (E)} N_e^\pm(E,y)d\H^{d-1}(y)
\]
where $\pi_{e^\perp}$ is the orthogonal projection on $e^\perp$ and where
\begin{align*}
N_e^+(E,y)&=\#\big(\pi_{e^\perp}^{-1}(y)\cap \partial_e^+E\big)
\\
N_e^-(E,y)&=\#\big(\pi_{e^\perp}^{-1}(y)\cap \partial_e^-E\big).
\end{align*}
In particular both $N_e^\pm(E,y)$ are finite for $\H^{d-1}$-a.e. $y\in \pi_{e^\perp}(E)$.

Now we claim that for $\H^{d-1}$-a.e. $y\in \pi_{e^\perp}(E)$ we have $N_e^+(E,y)=N_e^-(E,y)$. Indeed, by the area formula we also have $\pi_{e^\perp}(\partial_e^{\,0} E)=0$, which means that for $\H^{n-1}$-a.e. $y\in \pi_{e^\perp}(E)$ the line $\pi_{e^\perp}^{-1}(y)$ intersects $\partial E$ always where the normal has nonzero scalar product with $e$. As observed above, both cardinalities are finite for $\H^{d-1}$-a.e. $y$, and by Fubini for $\H^{d-1}$-a.e. $y$ the $1$-dimensional measure of $E\cap \pi_{e^\perp}(y)$ is finite. It follows that $E\cap \pi_{e^\perp}^{-1}(E)$ is a finite union of bounded intervals, hence we conclude that 
\[
N_e^+(E,y)=N_e^-(E,y)\quad\text{ for $\H^{d-1}$-a.e. $y$}
\]
as wanted.

We thus obtain
\begin{align*}
\int_{\partial E} \phi_\mu(\nu(x))d\H^{d-1}(x) &=\int_{\partial E}\left(\int_{\S^{d-1}}\langle \nu(x),v\rangle_+d\mu(v)\right)d\H^{d-1}(x)\\
&=\int_{\S^{d-1}}\left(\int_{\partial E}\langle \nu(x),v\rangle_+d\H^{d-1}(x)\right)d\mu(v)\\
&=\int_{\S^{d-1}}\left(\int_{\pi_{e^\perp}} N_e^+(E,y) d\H^{d-1}(y)\right)d\mu(v)\\
&=\int_{\S^{d-1}}\left(\int_{\pi_{e^\perp}} N_e^-(E,y) d\H^{d-1}(y)\right)d\mu(v)\\
&=\int_{\S^{d-1}}\left(\int_{\partial E}\langle \nu(x),-v\rangle_+d\H^{d-1}(x)\right)d\mu(v)\\
&=\int_{\S^{d-1}}\left(\int_{\partial E}\langle \nu(x),v\rangle_+d\H^{d-1}(x)\right)d\widetilde\mu(v).
\end{align*}
This proves the result for smooth, bounded sets.

To prove the general case, let $E$ be a finite perimeter set. Then there exists a sequence $E_j$ of bounded, smooth sets such that $E_j\to E$ in $L^1$ and $P(E_j)\to P(E)$ as $j\to \infty$. By Reshetnyak's theorem \cite[Theorem 2.38]{AFP} the same convergence holds for the perimeters defined by $\phi_\mu$ and $\phi_{\widetilde \mu}$. As a consequence
\[
P_{\phi_\mu}(E)=\lim_{j\to\infty}P_{\phi_\mu}(E_j)=\lim_{j\to\infty}P_{\phi_{\widetilde\mu}}(E_j)=P_{\phi_{\widetilde\mu}}(E).
\]
\end{proof}

\section{Final remarks}\label{sec:finalrem}
We collect here some direct generalizations, and some open questions and open directions that would require an essential extra effort compared to the present work, but which could be attacked in future work.
\subsection{\texorpdfstring{$\Gamma$}{Gamma}-limits for general signed potentials}
The main question that we leave open is to devise a precise, if possible necessary and sufficient, condition on general signed potentials $V$, so that the $\Gamma$-limit result of Theorem \ref{thm:main} (or the more general versions of Theorem \ref{thm:main} present in \cite{Gel, BraGel, AliGel} mentioned before) still hold.

\medskip

The following stability condition on the energy of finite configurations $X$ is a natural hypothesis to impose, when modelling crystalline Wulff shapes:
\[
    \exists C>0,\quad\forall X\subset \mathcal L,\quad   \mathcal E_V(X)\ge -C\ \sharp X.
\]
It is natural to study the $\Gamma$-limits analogous to our setting, but for general signed $V$, under the above hypothesis, or under a suitable strengthening of this hypothesis.

\subsection{Longer range interactions in quasicrystals}
There are several possible ways to extend the range of applicability of Theorem \ref{thmqc} beyond nearest-neighbor interactions between tiles of the tessellation induced by the multigrid construction of quasicrystals:
\begin{enumerate}
    \item In the dual space, associate an energy $W(y-x)$ depending on differences of positions of multigrid points $x,y$ belonging to a given fixed lattice, given by any $d$-ple of hyperplane families from the multigrid construction.
    \item Define the interaction $W(x,y)$ between two tiles in the primal space as a function of the sequence of rail directions required to reach $y$ starting from $x$. For example, in the case of next-nearest-neighbor interactions, let $W(x,y)$ be nonzero only if tile $y$ is a neighbor of $x$ or a neighbor of a neighbor of $x$. Restricted to neighbors we define $W$ as before, and if there exists $z$ such that $x,z$ belong to a multigrid rail corresponding to multiindex $J'$ and $z,y$ are neighbors corresponding to multiindex $\widetilde J'$, then define $W(x,y)$ as a function depending only on the tuple $(J',\widetilde J')$, and independent on the particular choices of $x,y$.
    \item Consider a potential $W:\mathbb R^d\to [0,+\infty)$ and if $x,y$ are tile centers from the multigrid tiling, define the interaction of the corresponding tiles to be $W(y-x)$.
\end{enumerate}
The first option is the only one in which a direct extension of the setup of Theorem \ref{thmqc} will prove a corresponding $\Gamma$-convengence result, however we do not develop the theory in this direction because the results would be notationally more involved and we do not have a direct application at hand. The remaining options instead seem better physically motivated by a model of quasicrystalline materials, however require an essential effort beyond the setup of Theorem \ref{thmqc} considered here. The comparison between these latter cases seems suited a thorough study in future work.

\subsection{Curved multigrids}
The construction based on merely straight (hyperplane) multigrids that we pursue here can be perturbed, replacing families of hyperplanes by families of curves with richer allowed behavior, which can be well-approximated by hyperplanes at large scale. One then constructs a tiling from the graph of intersections between $(d-1)$-ples of surfaces corresponding to different families, by an algorithm similar to Section \ref{sec_qctiling}. The extension of the proof of Theorem \ref{thmqc} to this setup seems to us an interesting open direction. On the one hand we mention the use of curved multigrids in modelling (see e.g. \cite[Sec. 2.3.2]{janssenbook}) and, on the other hand, on the mathematical side, it is interesting to find outthe correct oscillation bounds and metrics for the multigrid hypersurfaces. More precisely, it would be interesting to pursue the determination of natural sufficient regularity conditions on the families of surfaces that we would use, under which extensions of the result of Section \ref{sec_bddist} hold true.

\bibliographystyle{plain}
\bibliography{quasicrystals}

\end{document}